\documentclass[smallextended, envcountsame, numbook]{svjour3}
\usepackage{amsmath,amssymb, latexsym,amsfonts,natbib}
\usepackage[dvipdf]{graphicx}
\usepackage{color}

\smartqed

\newcommand{\R}{\mathbb{R}}

\newcommand{\B}{\mathbb{B}}

\newcommand{\CEV}{\operatorname{\it CEV}}

\def\eps{\varepsilon}

\def\Min(#1,#2){#1\wedge #2}
\def\Max(#1,#2){#1\vee #2}

\def\e{{\rm e}}

\def\vto{{\stackrel{v}{\to}}}

\allowdisplaybreaks

\begin{document}

\title{Conditional Extreme Value Models: Fallacies and Pitfalls}
\author{Holger Drees \and Anja Jan{\ss}en 
}
\institute{H. Drees \at
              University of Hamburg, Department of Mathematics,
 SPST, Bundesstr.\ 55, 20146 Hamburg, Germany. \\
              \email{drees@math.uni-hamburg.de}
              \and
A. Jan{\ss}en \at
University of Copenhagen, Department of Mathematics, Universitetsparken 5, 2100 Copenhagen, Denmark. \\
\email{anja@math.ku.dk}}

\maketitle

\begin{abstract}
Conditional extreme value models have been introduced by \cite{HR07} to describe the asymptotic behavior of a random vector as one specific component becomes extreme. Obviously,  this class of models is related to classical multivariate extreme value theory which describes the behavior of a random vector as its norm (and therefore at least one of its components) becomes extreme. However, it turns out that this relationship is rather subtle and sometimes contrary to intuition. We clarify the differences between the two approaches with the help of several illuminative (counter)examples. Furthermore, we discuss  marginal standardization, which is a useful tool in classical multivariate extreme value theory but, as we point out, much less straightforward and sometimes even obscuring in conditional extreme value models. Finally, we indicate how, in some situations, a more comprehensive characterization of the asymptotic behavior can be obtained if the conditions of conditional extreme value models are relaxed so that the limit is no longer unique.
\end{abstract}

\section{Introduction}
\subsection{Motivation and Overview}

The analysis of the extremal behavior of random vectors constitutes the central topic of multivariate extreme value theory (MEVT). Its historical development started with the extension of results from univariate extreme value theory to the multivariate setting by describing the limiting behavior of componentwise maxima of iid random vectors under suitable linear normalization. This behavior is determined by the distribution of $(X_1, \ldots, X_d)$ on the set of points with at least one large component. However, there are situations where this approach does not work, e.g., because the normalized maxima of some of the components do not converge.

In such a case, if one is interested  in the behavior of the other components as one component  becomes extreme, a so-called conditional extreme value model (CEVM) introduced by \cite{HR07} may be appropriate. This model describes (certain aspects of) the behavior of $(X_1, \ldots, X_d)$ on the set of points where a {\em given} component is large. More concretely, it is assumed that one marginal distribution function, say the $i$th, belongs to the domain of attraction of some extreme value distribution, and that the conditional distribution of the other components, given that $X_i$ {\em exceeds} a high threshold converges after a suitable linear normalization.
\cite{DR11a} and \cite{DR11b} explored some connections between a CEVM and the aforementioned model from classical MEVT  by comparing vague convergence of certain measures related to the distribution of $(X_1,\ldots,X_d)$ on different spaces. We use this framework of vague convergence in  Section \ref{Sec:EVTmodels} below, where we recall the basics of both concepts.

A somewhat related model has been suggested by \cite{HT04}. Here the basic assumption is that the conditional distribution of the other normalized components converge given that $X_i$ {\em is equal to} a large value. \cite{RZ14} rephrased this assumption in terms of convergence of Markov kernels and examined its relation to a CEVM.   They showed that, in the general case of non-standardized components, additional assumptions about the normalizing functions have to be satisfied in order to ensure that a CEVM follows from the convergence of Markov kernels, cf.\ \cite{RZ14}, Proposition 4.1. Since our analysis will be based on vague convergence of measures, we will take the CEVM assumption as a starting point instead of the model from \cite{HT04}.

In order to separate the marginal tail behavior from the dependence structure between extreme values of the components, it is common in classical MEVT to analyze the asymptotic behavior of the standardized vector $(X^*_1, \ldots, X_d^*)=(1/(1-F_1(X_1)), \ldots, 1/(1-F_d(X_d)))$, whose marginal distributions are  standard Pareto if the marginal distributions of $(X_1, \ldots, X_d)$ are continuous.  It is known that the componentwise maxima of the standardized vector converge  if and only if $tP\{(X^*_1/t, \ldots, X^*_d/t) \in \cdot \}$ converges vaguely to a non-degenerate limit measure, which means that $(X^*_1, \ldots, X_d^*)$ exhibits standard multivariate regular variation. Note that here the normalization is particularly simple and does not depend on the distribution of $(X_1^*,\ldots,X^*_d)$ anymore.

In the context of a CEVM, \cite{HR07} and \cite{DR11a} also discuss the issue of standardizing the marginal distributions. However, their aim was merely to ensure that the same simple normalization can be used, whereas they do not try to transform the marginal distribution to a prescribed one. We will see in Section \ref{Sec:Standardization} that the meaning and the interpretation of standardization in a CEVM is so far not well understood. In particular it turns out  that often a standardization  in the sense of \cite{HR07} and \cite{DR11a} completely changes the information about the extremal dependence structure. We thus think that
there does not exist a ``natural'' way to standardize the margins in a general CEVM. In some cases, however, when the CEVM describes the asymptotic behavior of the not necessarily extreme components locally in a one-sided neighborhood of a fixed point, standardization is feasible.

In Section \ref{Sec:Relations}, we discuss the aforementioned connections between the CEVM and classical models from MEVT in detail and point out that the relations are much more intricate than realized so far. We show by (counter)examples that despite  apparent similarities between the two models both concepts are too different to generally infer the convergence assumed in one model from the convergence in the other model - often contrary to intuition. For example, if the vector $(X_1,X_2)$ is in the domain of attraction of a bivariate extreme value distribution and the two components are not asymptotically independent, then one may conjecture that this vector satisfies a CEVM as well. However, since the  treatment of both the left tail and the central part of the distribution of $X_1$ differs in the two models, this needs not be the case. The counterexamples that we present show that, more fundamentally, both types of models may convey very different information about the behavior of $X_1$ for large values of $X_2$, because in the classical extreme value models the normalizing functions for $X_1$ are determined by its tail behavior while in a CEVM they must match the overall behavior of $X_1$ for large values of $X_2$, which need not be related to large values of $X_1$. One may try to overcome these differences by additional assumptions on the normalizing functions; cf.\ \cite{DR11a}. However, our counterexamples also show that some of the assumptions stated in the literature so far are in fact too weak in certain cases.

Finally, in Section \ref{sect:discussion} we put possible modifications of the CEVM forward which could help to overcome some of the drawbacks of the model identified in the preceding sections.
\section{Extreme Value Models}\label{Sec:EVTmodels}
\subsection{Classical MEVT}\label{Sec:MEVT}
Classical MEVT describes the limiting behavior of the joint distribution of componentwise maxima of iid random vectors after a suitable linear standardization of the margins. For notational simplicity, we focus on the bivariate case where for iid random vectors $(X_i,Y_i)$, $1\le i\le n$, it is assumed that
\begin{equation} \label{eq:MEVT}
 P\Big\{ \Big(\frac{\max_{1\le i\le n} X_i-b_n}{a_n},
 \frac{\max_{1\le i\le n} Y_i-d_n}{c_n}\Big)\in \cdot\Big\} \;\to\; G\quad \text{weakly}
\end{equation}
 for suitable normalizing constants $a_n,c_n>0$ and $ b_n, d_n\in\R$ and a bivariate distribution function $G$ with non-degenerate marginal distributions.

In particular, the marginal distribution functions $F_X$ and $F_Y$  of a random vector $(X,Y)$ with the same distribution as $(X_1,Y_1)$ belong to the domain of attraction of some extreme value distribution $G_{\gamma_X}$ resp.\ $G_{\gamma_Y}$, where $G_\gamma(x):=\exp\big(-(1+\gamma x)^{-1/\gamma}\big)$ for all $x$ such that $1+\gamma x>0$. This means that for suitable normalizing functions $a,c>0$ and $b,d\in\R$
\begin{eqnarray}
tP\Big\{\frac{X-b(t)}{a(t)}>x\Big\} & \to & (1+\gamma_Xx)^{-1/\gamma_X} \quad \forall\, x\in E^{(\gamma_X)}  \label{eq:margconvX}\\
tP\Big\{\frac{Y-d(t)}{c(t)}>y\Big\} & \to & (1+\gamma_Yy)^{-1/\gamma_Y} \quad \forall\, y\in E^{(\gamma_Y)} \label{eq:margconvY}
\end{eqnarray}
as $t\to\infty$. (Throughout the paper, $(1+\gamma x)^{-1/\gamma}$ is interpreted as $\e^{-x}$ if $\gamma=0$; likewise $(x^\gamma-1)/\gamma:=\log x$ for $\gamma=0$.)
Here
$$ E^{(\gamma)}:= \{x\in\R\mid 1+\gamma x>0\} = \left\{
   \begin{array}{lcl}
     (-\infty,-1/\gamma) & & \gamma<0,\\
     (-\infty,\infty) & \text{if} & \gamma=0,\\
     (-1/\gamma,\infty) & & \gamma>0
   \end{array}
   \right.$$
denotes the interior of the support of $G_\gamma$. Moreover, let
$$ q_\gamma:= \inf E^{(\gamma)} = \left\{
  \begin{array}{l@{\quad}l}
      -\infty & \gamma\le 0,\\
      -1/\gamma & \gamma>0
  \end{array}
  \right.
  \quad \text{and} \quad
  q^{\gamma}:= \sup E^{(\gamma)} = \left\{
  \begin{array}{l@{\quad}l}
      -1/\gamma & \gamma< 0,\\
      \infty & \gamma\ge 0
  \end{array}
  \right.
$$
be the left and right endpoint of $E^{(\gamma)}$.
Let $\bar E^{(\gamma)}=E^{(\gamma)}\cup\{q^{\gamma}\}$ denote the closure to the right of $E^{(\gamma)}$ (considered as a subset of the compactification $\bar\R=[-\infty,\infty]$ of $\R$), and let $\bar{\bar E}^{(\gamma)}=E^{(\gamma)}\cup\{q_\gamma,q^{\gamma}\}$ denote the topological closure of $E^{(\gamma)}$ in $\bar\R$. Define the one-point uncompactification of $\bar{\bar E}^{(\gamma_X)}\times \bar{\bar E}^{(\gamma_Y)}$ by  $E^{(\gamma_X,\gamma_Y)}:= \big(\bar{\bar E}^{(\gamma_X)}\times \bar{\bar E}^{(\gamma_Y)}\big)\setminus \{(q_{\gamma_X},q_{\gamma_Y})\}$.
See \cite{Res07}, Section 3.3, and \cite{DR11a} for details about the topology on these sets.

Convergence \eqref{eq:MEVT} is equivalent to the vague convergence
\begin{equation} \label{eq:MEVvague}
 t P\Big\{\Big(\max\Big(\frac{X-b(t)}{a(t)}, q_{\gamma_X}\Big) ,\max\Big(\frac{Y-d(t)}{c(t)}, q_{\gamma_Y}\Big)\Big)\in\cdot\Big\} \,\vto\, \mu \quad \text{as } t\to\infty
\end{equation}
on $\big([q_{\gamma_X},\infty]\times [q_{\gamma_Y},\infty]\big)\setminus \{(q_{\gamma_X},q_{\gamma_Y})\}$ for some normalizing functions $a,c>0$ and $b,d\in\R$ to a  Radon measure $\mu$ with non-degenerate margins satisfying $\mu(\{\infty\}\times [q_{\gamma_Y},\infty])=0=\mu([q_{\gamma_X},\infty]\times\{\infty\})$ (see \cite{BGST04}, (8.71)). This is equivalent to
\begin{equation} \label{eq:MEVconv1}
   t P\Big\{\Big(\max\Big(\frac{X-b(t)}{a(t)}, q_{\gamma_X}\Big) ,\max\Big(\frac{Y-d(t)}{c(t)}, q_{\gamma_Y}\Big)\Big)\in A\Big\}\to \mu(A)<\infty
\end{equation}
for all Borel sets $A\subset E^{(\gamma_X,\gamma_Y)}$ bounded away from $(q_{\gamma_X},q_{\gamma_Y})$ such that $\mu(\partial A)=0$. Note that for $A=(x,\infty)\times \bar{\bar E}^{(\gamma_Y)}$ and $A=\bar{\bar E}^{(\gamma_X)}\times (y,\infty)$ we recover the convergence of the left-hand sides of \eqref{eq:margconvX} and \eqref{eq:margconvY}, respectively. Hence, for the above choice of the normalizing functions, $\mu\big((x,\infty]\times \bar{\bar E}^{(\gamma_Y)}\big)=(1+\gamma_Xx)^{-1/\gamma_X}$ for all $x\in E^{(\gamma_X)}$, and  $\mu\big(\bar{\bar E}^{(\gamma_X)}\times (y,\infty]\big)=(1+\gamma_Yy)^{-1/\gamma_Y}$ for all $y\in E^{(\gamma_Y)}$.
Moreover, \eqref{eq:MEVvague} is equivalent to
\begin{equation} \label{eq:MEVconvsurv}
 t P\Big\{\frac{X-b(t)}{a(t)}>x \text{ or } \frac{Y-d(t)}{c(t)}> y\Big\} \,\to\, \mu \big(E^{(\gamma_X,\gamma_Y)}\setminus([q_{\gamma_X},x]\times [q_{\gamma_Y},y])\big)
\end{equation}
for all $(x,y)\in E^{(\gamma_X)}\times E^{(\gamma_Y)}$. Therefore, in a slight abuse of notation, one may also write
$$ t P\Big\{\Big(\frac{X-b(t)}{a(t)},\frac{Y-d(t)}{c(t)}\Big)\in\cdot\Big\} \,\vto\, \mu \quad \text{as } t\to\infty
$$
on $E^{(\gamma_X,\gamma_Y)}$.

Sometimes it is useful to examine the marginal distributions and the dependence structure separately. To this end, define marginally normalized random variables
$$ X^* := \frac 1{1-F_X(X)}, \quad Y^* := \frac 1{1-F_Y(Y)}.$$
 If $F_X$ is continuous, then $X^*$ has a standard Pareto distribution; more generally, under $\eqref{eq:MEVT}$, $X^*$ is tail-equivalent to a standard Pareto random variable, i.e.\ $xP\{X^*>x\}\to 1$ as $x\to\infty$. Now convergence \eqref{eq:MEVT} is equivalent to the marginal convergences
\eqref{eq:margconvX} and \eqref{eq:margconvY}
 together with the vague convergence
\begin{equation} \label{eq:MEVstandvague}
 t P\Big\{\Big(\frac{X^*}{t},\frac{Y^*}{t}\Big)\in\cdot\Big\} \,\vto\, \mu^* \quad \text{as } t\to\infty
\end{equation}
in $[0,\infty]^2\setminus\{(0,0)\}$ to the non-degenerate Radon measure $\mu^*$ given by
$$\mu^*\big([0,\infty]^2\setminus([0,x]\times[0,y])\big):= \mu\bigg(E^{(\gamma_X,\gamma_Y)}\setminus\Big( \Big[q_{\gamma_X},\frac{x^{\gamma_X}-1}{\gamma_X}\Big]\times  \Big[ q_{\gamma_Y},\frac{y^{\gamma_Y}-1}{\gamma_Y}\Big]\Big)\bigg)
$$
for all $x,y>0$.
The latter convergence means that
$tP\{(X^*,Y^*)\in tA\}\to  \mu^*(A)$ for all Borel sets bounded away from the origin that are continuity sets of $\mu^*$. Hence, the limit relations \eqref{eq:MEVvague} and \eqref{eq:MEVstandvague} describe the asymptotic behavior of $(X,Y)$ and $(X^*,Y^*)$, respectively, when {\em at least one} component of the random vector is large. Note that the limit measure $\mu^*$ is homogeneous of order $-1$, i.e. $\mu^*(tA)=t^{-1}\mu^*(A)$ for all $t>0$ and $A\in\B([0,\infty)^2)$.

 \subsection{Conditional Extreme Value Models}\label{Sec:CEVM}
In some applications, though, one is interested in the behavior of the random vector if a {\em given} component, say $Y$, is large. For example, the
marginal expected shortfall of an asset which is defined as the expected loss from this investment given that some other quantity (e.g., the total loss of a portfolio) exceeds a high threshold is an important risk measure for a  manager who must decide whether to add the asset to a given portfolio.  Another example is given by internet traffic data, as analyzed in \cite{DR11b}, who looked at both file sizes and transfer rates of transmissions in a network. They argued that the distribution of the transfer rates may not be in the domain of attraction of an extreme value distribution, in contrast to the file size distribution. Still, the asymptotic behavior of the transmission rate, given that the size of this transmitted file is large, is an important performance measure of the network. See also \cite{HT04} for an example concerning air pollutants.

In situations like these, the asymptotic behavior of the random vector may be described by a so-called conditional extreme value model (CEVM), as introduced in \cite{HR07} and further developed in \cite{DR11a}.

Here we start with a clarification of the model assumptions given by \cite{RZ14}.

\begin{definition} \label{def:CEV}

  The vector $(X,Y)$ follows a CEVM if there exist normalizing functions $\alpha,c>0$ and $\beta,d\in\R$ such that \eqref{eq:margconvY} holds and, as $t\to\infty$,
  \begin{equation} \label{eq:CEVconv}
    tP\Big\{\Big(\frac{X-\beta(t)}{\alpha(t)},\frac{Y-d(t)}{c(t)}\Big)\in\cdot\Big\} \,\vto\, \mu_{X,Y>}(\cdot)
  \end{equation}
  vaguely on $\bar\R\times \bar E^{(\gamma_Y)}$, and the limit measure $\mu_{X,Y>}$ satisfies the following non-degeneracy conditions:
  \begin{enumerate}
   \item $\mu_{X,Y>}(\cdot\times(y,\infty])$ is a  non-degenerate measure on $\bar\R$ for all $y\in E^{(\gamma_Y)}$, i.e.\ $\mu_{X,Y>}\big((\bar\R\setminus\{x\})\times(y,\infty]\big)>0$ for all $x\in\bar\R$ and $y\in E^{(\gamma_Y)}$;
   \item $\mu_{X,Y>}(\{\infty\}\times \bar E^{(\gamma_Y)})=0$.
  \end{enumerate}
  Then we write in short $(X,Y)\in \CEV(\alpha,\beta,c,d,\gamma_Y,\mu_{X,Y>})$.
\end{definition}
Since $[-\infty,\infty]$ is compact, the vague convergence \eqref{eq:CEVconv} describes the behavior of the random vector  $(X,Y)$ when only $Y$ needs to be extreme.

\begin{remark}\label{rem:CEV}
\begin{enumerate}
  \item Here and in the sequel, we extend a measure $\nu$ defined on a Borel subset $M$ of $\bar \R^2$ to a measure $\tilde\nu$ on $\bar\R^2$ by $\tilde\nu(A):=\nu(A\cap M)$ if necessary. For simplicity, in what follows we denote the extension again by $\nu$.
  \item Similarly as above, \eqref{eq:CEVconv} is equivalent to
     $$  tP\Big\{\Big(\frac{X-\beta(t)}{\alpha(t)},\frac{Y-d(t)}{c(t)}\Big)\in A\Big\} \to \mu_{X,Y>}(A)
     $$
     for all Borel sets $A\subset \bar\R\times(q_{\gamma_Y},\infty]$ with $\inf\{y|(x,y)\in A\}>q_{\gamma_Y}$ such that $\mu_{X,Y>}(\partial A)=0$. This is equivalent to
     $$  tP\Big\{\frac{X-\beta(t)}{\alpha(t)}\le x,\frac{Y-d(t)}{c(t)}>y\Big\}\,\to\, \mu_{X,Y>}([-\infty,x]\times (y,\infty])
     $$
     for all $x\in\bar\R$ and $y\in E^{(\gamma_Y)}$  with $\mu_{X,Y>}(\partial ([-\infty,x]\times (y,\infty]))=0$.
  \item
  \cite{DR11a} and \cite{RZ14} do not assume \eqref{eq:margconvY}, but they {\em conclude} this convergence for {\em some}  extreme value index $\tilde\gamma_Y\in\R$ from the remaining assumptions. Note, however, that this index need not coincide with the  value $\gamma_Y$ used in the definition of the CEVM which only determines the support of $\mu_{X,Y>}$. In particular, if the assumptions of Definition \ref{def:CEV} are fulfilled for some $\gamma_Y>0$, then the convergence of measures \eqref{eq:CEVconv} obviously holds true on $\bar E^{(\gamma_Y')}$ for all  $\gamma_Y'>\gamma_Y$, too. Therefore, we have chosen the present formulation to avoid ambiguities concerning the meaning of the parameter $\gamma_Y$. Moreover, this way it seems more natural to consider vague convergence on the space $[-\infty,\infty]\times \bar E^{(\gamma_Y)}$.

  The additional normalizing condition $\mu_{X,Y>}(\bar\R\times(0,\infty])=1$ required by \cite{DR11a} and \cite{RZ14} follows readily from convergence \eqref{eq:margconvY}.
  \end{enumerate}
\end{remark}

Condition (ii) has been added by \cite{RZ14} to the original set of conditions used by \cite{HR07} and \cite{DR11a} to ensure uniqueness of the type of limit measure (i.e.\ uniqueness up to scale and shift transformations). Yet, as the following example shows, this condition has to be strengthened to rule out that different normalizations lead to completely different limit measures.
\begin{example} \label{ex:condii}
  Let $Y$ be a standard Pareto random variable (i.e. $P\{Y>x\}=1/x$ for all $x>1$) so that \eqref{eq:margconvY} holds with $\gamma_Y=1$. Let $B$ be an independent discrete random variable that is uniformly distributed on $\{0,1\}$,
  $$ X:= B+(1-B)(2-1/Y), $$
  and $c(t):=d(t):= t$.
  \begin{itemize}
    \item For $\beta(t)=0$ and $\alpha(t)=1$ one has, for all $y\in E^{(1)}=(-1,\infty)$, $x\in\R$ and sufficiently large $t$,
        \begin{eqnarray*}
          \lefteqn{tP\Big\{\frac{X-\beta(t)}{\alpha(t)}\le x,\frac{Y-c(t)}{d(t)}>y\Big\}}\\
          & = & \frac t2 \big( P\{1\le x, Y>t(1+y)\}+P\{2-1/Y\le x, Y>t(1+y)\}\big)\\
          & = & \frac t2 \Big( \frac 1{t(1+y)}1_{[1,\infty)}(x)+\frac 1{t(1+y)}1_{[2,\infty)}(x)\Big) \\
          & = & \frac 1{2} \big(1_{[1,\infty)}(x)+1_{[2,\infty)}(x)\big)(1+y)^{-1}\\
          & = & \mu_{X,Y>}\big([-\infty,x]\times (y,\infty]\big).
        \end{eqnarray*}
        Here $\mu_{X,Y>}=(\frac 12\delta_1+\frac12 \delta_2)\otimes\nu_1$, where $\delta_x$ denotes the Dirac-measure at $x$  and $\nu_1$ is given by $\nu_1(y,\infty]=1/(1+y)$ for all $y>-1$.
This verifies \eqref{eq:CEVconv}, and the limit measure satisfies the conditions (i) and (ii) of Definition \ref{def:CEV}.
    \item Choosing $\beta(t)=2$ and $\alpha(t)=1/t$ instead yields for all $x\in\R$ and $y\in E^{(1)}=(-1,\infty)$
        \begin{eqnarray*}
          \lefteqn{tP\Big\{\frac{X-\beta(t)}{\alpha(t)}\le x,\frac{Y-d(t)}{c(t)}>y\Big\}}\\
          & = & \frac t2 \big( P\{-1\le x/t, Y>t(1+y)\}+P\{1/Y\ge -x/t, Y>t(1+y)\}\big)\\
          & \to & \frac 1{2}(1+y)^{-1} + \frac 12 \big((1+y)^{-1}-x^-\big)^+\\
          & = & \tilde\mu_{X,Y>}\big([-\infty,x]\times (y,\infty]\big)
        \end{eqnarray*}
        as $t\to\infty$, with $x^-:=\max(-x,0)$.  Here, $\tilde \mu_{X,Y>}$ is the sum of $\frac 12\delta_{-\infty}\otimes\nu_1$ and the measure $\bar\mu$ which is concentrated on $\{(-(1+y)^{-1},y)|y>-1\}$ with $\bar\mu(\{(-(1+y)^{-1},y)|y>r\})=1/(2(1+r))$ for all $r>-1$. Again all conditions of Definition \ref{def:CEV} are met, but $\tilde \mu_{X,Y>}\big(\{-\infty\}\times(y,\infty]\big)=1/(2(1+y))>0$ for all $y\in E^{(1)}$.
    \end{itemize}
    So while the first limit measure describes the coarse behavior  of $X$ for large values of $Y$, namely that $X$ can attain only values near the points 1 and 2, the second limit specifies the behavior of $X$ in a neighborhood of 2 in greater detail, but it loses almost all information on the behavior of $X$ on $(-\infty,2-\eps]$ for any $\eps>0$. \qed
  \end{example}

To ensure that the convergence to types theorem can be applied which implies the uniqueness of the limit measure (up to scaling and shifts), we replace condition (ii) of Definition \ref{def:CEV} with
\begin{itemize}
 \item[(ii*)] $\mu_{X,Y>}\big(\{-\infty,\infty\}\times E^{(\gamma_Y)}\big)=0$.
\end{itemize}
If $(X,Y)$ satisfies this modified set of conditions, we write $(X,Y)\in \CEV^*(\alpha,\beta,c,d,\gamma_Y,\mu_{X,Y>})$. In Section 4 we discuss why in some situations it might be advantageous to drop condition (ii) (and (ii*)) although then the uniqueness of the limit measure is lost.
\begin{remark}
   One can avoid any assumption ruling out mass in lines through infinity by using the framework of $M$-convergence instead of vague convergence, cf.\ \cite{HuLi06} and \cite{LiReRo14}.  Since assumption (ii*) is sufficient for our subsequent analysis, here we stick to the more conventional approach in order to facilitate the comparison with the previous literature on CEVM.
\end{remark}


\section{Standardization of CEVM}\label{Sec:Standardization}

In classical MEVT one often considers the marginally standardized random vector $(X^*,Y^*)$, which has standard Pareto margins if the distribution functions of $X$ and $Y$ are continuous. This allows for a separation of the marginal tail behavior (investigated in univariate extreme value theory) and the extremal dependence structure which is inherent to MEVT. It seems reasonable to ask whether a similar separation can be achieved in CEVM, too.

It is easily seen that the component $Y$ assumed to be large can be standardized in the same way as before, because its distribution function belongs to the domain of attraction of $G_{\gamma_Y}$. Hence, convergence \eqref{eq:CEVconv} holds if and only if
\begin{equation} \label{eq:CEVYstandconv}
t P\Big\{ \Big( \frac{X-\beta(t)}{\alpha(t)}, \frac{Y^*}t\Big) \in \cdot \Big\} \,\vto\, \mu^*_{X,Y>}(\cdot)
\end{equation}
vaguely in $[-\infty,\infty]\times(0,\infty]$ with $\mu^*_{X,Y>}$ defined by
$$ \mu^*_{X,Y>}([-\infty,x]\times(y,\infty]) := \mu_{X,Y>}\Big([-\infty,x]\times\Big(\frac{y^{\gamma_Y}-1}{\gamma_Y},\infty\Big]\Big)
$$
for all $x\in\R$ and $y>0$; see \cite{DR11a}, Section 3, for details. (Note that in this context the normalizing functions for $Y^*$ are chosen slightly differently so that the limit in \eqref{eq:margconvY} equals $y^{-1}$ for all $y>0$; to obtain the usual limit, one has to consider $(Y^*-t)/t$ instead.)

In contrast, the distribution function of $X$ does in general not belong to the domain of attraction of any extreme value distribution. Indeed, not the tail behavior of $X$ is of interest, but the conditional behavior of $X$ when $Y$ attains large values. \cite{HR07} and \cite{DR11a} thus aimed at the more modest goal of finding a function $f$ such that $(f(X),Y^*)$ follows a CEVM with normalizing functions $\beta(t)=d(t)=0$ and $\alpha(t)=c(t)=t$, i.e.
\begin{equation} \label{eq:CEVstandconv}
t P\Big\{ \Big( \frac{f(X)}{t}, \frac{Y^*}t\Big) \in \cdot \Big\} \,\vto\, \mu^{**}_{X,Y>}(\cdot)
\end{equation}
vaguely in $[0,\infty]\times(0,\infty]$ with $\mu^{**}_{X,Y>}$ satisfying the non-degeneracy conditions
\begin{align}
   & \mu^{**}_{X,Y>}(\cdot\times (y,\infty])   \text{ is a finite non-degenerate measure for all } y>0, \text{ and } \label{eq:nondeg1}\\
   & \mu^{**}_{X,Y>}(\{\infty\}\times (0,\infty]) = 0  = \mu^{**}_{X,Y>}([0,\infty]\times  \{\infty\}). \label{eq:nondeg2}
\end{align}
Note that, unlike $X^*$ in classical MEVT, $f(X)$ does not have any pre-specified distribution.
Hence, the notion ``standardization'' only refers to the resulting normalizing functions, but not to the marginal distributions.

\cite{DR11a} required $f$ to meet the following conditions:
\begin{itemize}
  \item[(F1)] $f: \mbox{range}(X) \to (0,\infty)$ (with $\mbox{range}(X)$ denoting the range of $X$),
  \item[(F2)] $f$ is monotone,
  \item[(F3)] $f$ is unbounded.
\end{itemize}
The latter condition seems superfluous, because for bounded $f$ the limit measure in \eqref{eq:CEVstandconv} is necessarily concentrated on $\{0\}\times(0,\infty]$, violating \eqref{eq:nondeg1}.

The main claim in Section 3 of \cite{DR11a} is that such a function $f$ exists if and only if $\mu_{X,Y>}$, and thus $\mu^*_{X,Y>}$, is not a product measure. We show by counterexamples that in general neither of both implications hold. Moreover, even if both \eqref{eq:CEVconv} and \eqref{eq:CEVstandconv} hold, the respective limit measures may convey very different information about the conditional distribution of $X$ given that $Y$ is large.

\begin{example}  \label{ex:stand1}
  Let $Y$ be a standard Pareto random variable (and thus $Y=Y^*$), and for some independent discrete random variable $B$ that is uniformly distributed on $\{-1,1\}$ define $X:=2+B/Y$. Then, for all $x\in\R$ and $y>0$
  \begin{eqnarray*}
    \lefteqn{tP\Big\{\frac{X-2}{t^{-1}}\le x, \frac{Y^*}t>y\Big\} }\\
    & = & t \Big( P\Big\{B=1, Y\ge \frac tx, Y>ty\Big\} 1_{(0,\infty)}(x) \\
    & & \quad +  P\Big\{B=-1, Y\le \frac t{|x|}, Y>ty\Big\} 1_{(-\infty,0)}(x) + P\{B=-1, Y>ty\}1_{[0,\infty)}(x)\Big)\\
    & \to & \frac 12 \Big(\min\Big(x,\frac 1y\Big)1_{(0,\infty)}(x)+ \Big(\frac 1y-|x|\Big)^+1_{(-\infty,0)}(x) + \frac 1y 1_{[0,\infty)}(x)\Big)\\
    & = & \mu^*_{X,Y>}([-\infty,x]\times(y,\infty]),
  \end{eqnarray*}
  i.e., $(X,Y^*)$ satisfies the conditions of Definition \ref{def:CEV} (except for the normalized form of the limit in \eqref{eq:margconvY}) and (ii*) as well. Here $\mu^*_{X,Y>}$ denotes the measure that is concentrated on $\{(y^{-1},y)|y>0\}\cup \{(-y^{-1},y)|y>0\}$ with $\mu^*_{X,Y>}\big(\{(y^{-1},y)|y>r\}\big)=1/(2r)=\mu^*_{X,Y>}\big(\{(-y^{-1},y)|y>r\}\big), r>0$; hence it is not a product measure.

  Now consider an arbitrary function $f:(1,3)\to (0,\infty)$ satisfying (F1) and (F2). Since $f$ is monotone,  $u_0:=\sup_{3/2\le x\le 5/2} f(x)$ is finite. Thus, for $x,y>0$ and $t>\max(2/y,u_0/x)$,
  $$
     tP\Big\{\frac{f(X)}t>x,\frac{Y^*}t>y\Big\} = t P\big\{f(2+B/Y)>tx>u_0, Y>ty>2\big\} = 0.
  $$
  So, if \eqref{eq:CEVstandconv} holds, then $\mu^{**}_{X,Y>}((0,\infty)\times (y,\infty])=0$ for all $y>0$, that is,
  $\mu^{**}_{X,Y>}(\cdot\times (y,\infty])$ is degenerated, with all its mass concentrated at 0.

  We have hence shown that $X$ cannot be standardized in the sense used by \cite{DR11a}, despite the fact that $\mu^{*}_{X,Y>}$ is not a product measure. \qed
\end{example}

The next example shows that in some cases $X$ can even be standardized if $\mu^{*}_{X,Y>}$ is a product measure, and that the limit measures may describe completely different aspects of the stochastic behavior of $X$, given that $Y$ is large.
\begin{example} \label{ex:stand2}
 Let $X:=B(1-U/Y)$, for $Y$ and $B$ as in Example \ref{ex:stand1} and an independent random variable $U$ that is uniformly distributed on $(0,1)$. Then, for all  $x\neq -1, y>0$ and $t>1/y$,
 \begin{eqnarray*}
  t P\Big\{X\le x, \frac{Y^*}t >y\Big\} & = & \frac t2 \big( P\{1-U/Y\le x,Y>ty\}+P\{U/Y-1\le x, Y>ty\}\big) \\
  & = & \frac t2 E\bigg( P\Big(ty<Y\le \frac U{1-x}\,\Big|\, U \Big)1_{(-\infty,1)}(x)+P\{Y>ty\}1_{[1,\infty)}(x) \\
  & & \hspace*{1cm} +
  P\Big(Y> \max\Big(\frac U{1+x},ty\Big)\,\Big|\, U \Big)1_{(-1,0)}(x)+P\{Y>ty\}1_{[0,\infty)}(x)\bigg) \\
  & = & \frac t2 \bigg( E\Big( \Big(\frac 1{ty}-\frac{1-x}U\Big)^+\Big)1_{(-\infty,1)}(x)+ \frac 1{ty}1_{[1,\infty)}(x) \\
  & & \hspace*{1cm}
   + E\Big( \min\Big(\frac 1{ty},\frac{1+x}U\Big)\Big)1_{(-1,0)}(x) +\frac 1{ty}1_{[0,\infty)}(x)\bigg)\\
  & = &  E\Big( \Big(\frac 1{2y}-\frac{t(1-x)}{2U}\Big)^+\Big)1_{(-\infty,1)}(x)+\frac 1{2y}1_{[1,\infty)}(x) \\
  & & \hspace*{1cm}
 + E\Big( \min\Big(\frac 1{2y},\frac{t(1+x)}{2U}\Big)\Big)1_{(-1,0)}(x) +\frac 1{2y}1_{[0,\infty)}(x)\\
  & \to & \frac 1{2y}1_{[1,\infty)}(x)+\frac 1{2y}1_{(-1,0)}(x)+\frac 1{2y}1_{[0,\infty)}(x)\\
  & = & \frac 1{2y}\big(1_{[1,\infty)}(x)+1_{[-1,\infty)}(x)\big)\\
   & = &  \mu^*_{X,Y>}([-\infty,x]\times (y,\infty])
 \end{eqnarray*}
 as $t\to\infty$. Thus, \eqref{eq:CEVYstandconv} holds with limit measure $\mu^*_{X,Y>}= (\frac 12\delta_{-1}+\frac 12\delta_{1})\otimes\tilde\nu_1$, where $\tilde \nu_1$ is given by $\tilde \nu_1(y,\infty]=1/y$ for all $y>0$.

 Now define $f:(-1,1)\to (0,\infty)$ by $f(x)=1/(1-x)\in (1/2,\infty)$, which obviously satisfies (F1)--(F3). For all $x,y>0$ and $t>\max(1/x,1/y)$, one has
 \begin{eqnarray*}
   tP\Big\{\frac{f(X)}t\le x, \frac{Y^*}t>y\Big\} & = & tP\big\{X\le 1-1/(tx),Y>ty\big\} \\
   & = & \frac t2 \Big(P\Big\{1-\frac UY\le 1-\frac 1{tx},  Y>ty\Big\}  + P\{Y>ty\}\Big)\\
   & = & \frac t2 E\big( P(ty<Y \le tUx\,|\,U)\big) + \frac 1{2y}\\
   & = & \frac 12 \int_{(y/x,1]} \frac 1y-\frac 1{ux}\, du + \frac 1{2y}\\
   & = & \frac 12 \Big(\frac 1y-\frac 1x+\frac{\log(y/x)}x\Big)1_{\{x>y\}}+ \frac 1{2y}.
 \end{eqnarray*}
 Direct calculations show that the last expression equals $\mu^{**}_{X,Y>}([0,x]\times(y,\infty])$ where the measure $\mu^{**}_{X,Y>}$ is the sum of $\frac 12 \delta_0\otimes\tilde{\nu}_1$ and the measure with Lebesgue density $g(x,y)=(2x^2y)^{-1}1_{\{x>y\}}$.
 Hence, \eqref{eq:CEVstandconv} holds with a non-degenerate limit measure $\mu^{**}_{X,Y>}$, although $\mu^{*}_{X,Y>}$ is a product measure.

 Note that the limit measures $\mu^*_{X,Y>}$ and $\mu^{**}_{X,Y>}$ convey very different information on the behavior of $X$. While the former shows that,  for large values of $Y$, the random variable $X$ may only attain values close to $\pm 1$ (with probability $1/2$ in each case), but does not reveal any more detailed information on its behavior in the vicinity of these points, $\mu^{**}_{X,Y>}$ describes the fine structure of the conditional distribution of $X$ near 1, but does not give any information about its behavior on sets bounded away from 1, beyond the fact that half of its mass is concentrated on these sets. \qed
 \end{example}

 \begin{remark}
   In Example \ref{ex:stand2} the measure $\mu^{**}_{X,Y>}(\cdot\times(y,\infty])$ has half of its mass concentrated in 0. In a sense, this point corresponds to $-\infty$ in the original setting for $(X,Y)$. Hence, one might think of strengthening condition \eqref{eq:nondeg2} to $\mu^{**}_{X,Y>}(\{0,\infty\}\times (0,\infty]) = 0 = \mu^{**}_{X,Y>}([0,\infty]\times  \{\infty\})$, similarly as we have replaced condition (ii) in Definition \ref{def:CEV} with (ii*). Then counterexamples of the type considered in Example \ref{ex:stand2} are ruled out, and the conclusion that a standardization of $X$ is impossible if $\mu^*_{X,Y>}$ is a product measure may be correct. Note, however, that (ii*) has been introduced to ensure that the convergence to types theorem is applicable and the limit measure is unique (up to scale and shift transformations). In the situation of \eqref{eq:CEVstandconv}, where we consider convergence on a space with finite left endpoint of the $x$-component, uniqueness of the limit is already ensured under the weaker condition \eqref{eq:nondeg2}.
\end{remark}

Conversely, one may construct random vectors $(X,Y)$ which do not fulfill the conditions of a CEVM (incl.\ condition (ii*)), but the relations \eqref{eq:CEVstandconv}--\eqref{eq:nondeg2} hold for a suitable function $f$ which satisfies (F1)--(F3), and the limit measure $\mu^{**}_{X,Y>}$ is not a product measure. The following example of this type is a modification of Example 4.1 of \cite{RZ14}.
\begin{example} \label{ex:stand3}
 Let $Y$ be standard Pareto, $X:=\e^Y$ and $f:=\log$. Then obviously \eqref{eq:CEVstandconv} holds with a measure $\mu^{**}_{X,Y>}((x,\infty]\times(y,\infty])=\min\big(x^{-1},y^{-1}\big)$ for all $x\ge 0, y>0$, which is concentrated on the main diagonal and hence cannot be a product measure.

 Suppose $(X,Y)=(X,Y^*)$ satisfies convergence \eqref{eq:CEVYstandconv} for some normalizing functions $\alpha>0$ and $\beta\in\R$ and some measure $\mu^*_{X,Y>}$ such that $\mu^*_{X,Y>}(\{-\infty,\infty\}\times (0,\infty])=0$. \cite{HR07} have shown in their Proposition 1 that then there exist $C,\rho\in\R$ such that, as $t\to\infty$,
  \begin{equation} \label{eq:normregvar}
    \frac{\alpha(\lambda t)}{\alpha(t)} \to \lambda^\rho,\quad \frac{\beta(\lambda t)-\beta(t)}{\alpha(t)} \to C \frac{\lambda^\rho-1}\rho.
  \end{equation}
   Here one may assume without loss of generality that $C\ne 0$; else replace $\beta(t)$ with $\beta(t)+\alpha(t)$ for which \eqref{eq:CEVYstandconv} also holds true (with a different limit measure). A combination of Theorem B.2.2 and Corollary B.2.13 of \cite{dHF06} shows that $|\beta|$ is regularly varying with index $\rho$ or $0$. Thus $\alpha(t)x+\beta(t)=o(t^{\max(\rho,0)+\eps})$  for all $x \in \mathbb{R}$, by Proposition B.1.9 of \cite{dHF06}. Then, however,
  \begin{eqnarray*}
    t P\Big\{\frac{X-\beta(t)}{\alpha(t)}\le x, \frac{Y^*}t>y\Big\}
    & = & tP\big\{Y\in(ty,\log(\alpha(t)x+\beta(t))]\big\} 1_{[\e,\infty)}(\alpha(t)x+\beta(t))\\
    & = & \Big(\frac 1y - \frac t{\log(\alpha(t)x+\beta(t))}\Big)^+ 1_{[\e,\infty)}(\alpha(t)x+\beta(t))
  \end{eqnarray*}
  converges to 0 for all $x\in\R$ and all $y>0$, contradicting the assumptions of a CEVM. \qed
\end{example}

The essential reason why standardization of $X$ fails in Example  \ref{ex:stand1} is that convergence  \eqref{eq:CEVstandconv} only describes the conditional behavior of $f(X)$, given $Y$ is large, on a {\em left} neighborhood of the ``point'' $\infty$, while convergence \eqref{eq:CEVYstandconv} specifies the conditional behavior of $X$ on a {\em two-sided} neighborhood of 2, which obviously cannot be mapped onto a one-sided neighborhood (of $\infty$) by a monotone transformation. In Example \ref{ex:stand2} the problem arises, because convergence \eqref{eq:CEVYstandconv} for $(X,Y^*)$ describes the conditional behavior of $X$ at the two different points $1$ and $-1$ (though only in a very crude way), while  \eqref{eq:CEVstandconv} can only convey information on the conditional behavior on a one-sided neighborhood of a single point (on a more detailed scale).

So the type of information given by the limit measure in a CEVM may be of qualitatively different nature from the one given by the limit measure \eqref{eq:MEVvague} in classical multivariate extreme value theory (or in the type of models considered by \cite{LT97}). The concept of standardization considered by \cite{HR07} and \cite{DR11a} does not make allowance for this crucial difference. It is therefore not suitable for CEVM in their full generality.

However, when the normalizing constants in the definition of the CEVM are such that the limit model focusses on the behavior of $X$ in a one-sided neighborhood of a single point (possibly $\pm\infty$), then a standardization is often possible. We do not aim at the most general results of that type, but merely discuss two important cases.

\begin{proposition} \label{prop:standposs}
  Suppose that $(X,Y)\in \CEV^*(\alpha,\beta,c,d,\gamma_Y,\mu_{X,Y>})$. If one of the following sets of conditions is fulfilled, then $X$ can be standardized, i.e.\ there is a function $f$ satisfying (F1) and (F2) such that \eqref{eq:CEVstandconv} holds:
  \begin{enumerate}
    \item $\alpha(t)\to\infty$ as $t\to \infty$, $\beta(t)=0$ and $\mu^*_{X,Y>}\big((0,\infty)\times (y,\infty)\big)>0$ for all $y>0$, or
    \item $\alpha(t)\to 0$ as $t \to\infty$, $\beta(t)=\beta_0$ for some $\beta_0\in\R$, and $\mbox{range}(X)\subset(-\infty,\beta_0)$ or $\mbox{range}(X)\subset(\beta_0,\infty)$.
  \end{enumerate}
\end{proposition}
\begin{proof}
  In the case (i), $\alpha$ is regularly varying with some index $\rho> 0$. (The case $\rho=0$ is excluded, since $\beta$ is not $\Pi$-varying; see \cite{HR07}, Subsection 2.2.)  W.l.o.g.\ we may assume that $\alpha$ is increasing and invertible with $\alpha(1)=1$ and increasing inverse function $\alpha^\leftarrow$, because there exists an asymptotically equivalent function $\tilde \alpha$  with these properties and $(X,Y)\in \CEV^*(\tilde\alpha,0,c,d,\gamma_Y,\mu_{X,Y>})$. Let $f(x):= \alpha^\leftarrow(x)$ for $x\ge 1$ and $f(x)=1/(2-x)$ for $x<1$ which is an increasing function. Then, for $x>0$ and $t>1/x$,
   \begin{eqnarray*}
    tP\Big\{ \frac{f(X)}t> x,\frac{Y^*}t>y\Big\}
      & = & tP\Big\{ X> \alpha(tx),\frac{Y^*}t>y\Big\} \\
      & = & tP\Big\{\frac{X}{\alpha(t)}> \frac{\alpha(tx)}{\alpha(t)},\frac{Y^*}t>y\Big\} \\
      & \to & \mu^*_{X,Y>}\big((x^\rho,\infty]\times (y,\infty]\big),
  \end{eqnarray*}
   if $\mu^*_{X,Y>}(\{x^\rho\}\times (y,\infty])=0$. Hence, \eqref{eq:CEVstandconv} holds with limit measure $\mu^{**}_{X,Y>}=(\mu^{*}_{X,Y>})^T$ induced by the transformation $T:\bar\R\times (0,\infty)\to \bar\R\times (0,\infty)$, $T(x,y)=((\max(x,0))^{1/\rho},y)$ (i.e.\ $\mu^{**}_{X,Y>}(\cdot)=\mu^{*}_{X,Y>}(T^{-1}(\cdot))$). The assumption $\mu^*_{X,Y>}\big((0,\infty)\times (y,\infty)\big)>0$ ensures that $\mu^{**}_{X,Y>}$ is not degenerate.

   The arguments are similar in the case (ii). Here, $\alpha$ is regularly varying with some index $\rho< 0$ and we may assume w.l.o.g.\ that $\alpha$ is decreasing and invertible. If $\mbox{range}(X)\subset(-\infty,\beta_0)$, let $f(x)=\alpha^\leftarrow(\beta_0-x)$ on a left neighborhood of $\beta_0$, else $f(x):=\alpha^\leftarrow(x-\beta_0)$ for $x$ in a right neighborhood of $\beta_0$. In the former case, $f$ is increasing, and, for $x\ge 0$ and sufficiently large $t$, we have
   \begin{eqnarray*}
    tP\Big\{ \frac{f(X)}t\le  x,\frac{Y^*}t>y\Big\}
      & = & tP\Big\{ X-\beta_0\le-\alpha(tx),\frac{Y^*}t>y\Big\} \\
      & = & tP\Big\{\frac{X-\beta_0}{\alpha(t)}\le -\frac{\alpha(tx)}{\alpha(t)},\frac{Y^*}t>y\Big\} \\
      & \to & \mu^*_{X,Y>}\big([-\infty,-x^\rho]\times (y,\infty]\big)
  \end{eqnarray*}
   if $\mu^*_{X,Y>}(\{-x^\rho\}\times (y,\infty])=0$. Note  that the assumptions imply that $\mu^*_{X,Y>}$ is concentrated on $(-\infty,0]\times(0,\infty)$. Hence, the limit defines a non-degenerate measure $\mu^{**}_{X,Y>}=(\mu^{*}_{X,Y>})^T$ with $T:[-\infty,0]\times (0,\infty)\to \bar\R\times (0,\infty)$, $T(x,y)=(|x|^{1/\rho},y)$. The case $\mbox{range}(X)\subset(\beta_0,\infty)$ can be treated analogously. \qed
\end{proof}
While in case (ii) the standardized CEVM conveys the same information about the conditional behavior of $X$ as the original CEVM, in  case (i) the information about the behavior of negative values of $X$ with large modulus is lost. (Of course, an analogous result holds if $\mu^*_{X,Y>}\big((-\infty,0)\times (y,\infty)\big)>0$ and one uses a suitable decreasing standardizing function $f$, so that the information about large values of $X$ is lost.) If one relaxes the conditions on the function $f$ in that one does not restrict its range, but allows $f$ to be $\R$-valued, then one can standardize $X$ in the case $\alpha(t)\to\infty$ and $\beta(t)=0$ without additional assumptions and without any loss of information. (To this end, define $f(x)=\alpha^\leftarrow(|x|)x/|x|$ for $x\ne 0$ and $f(0)=0$ for a version of the normalizing function $\alpha$ satisfying $\alpha(0,\infty)=(0,\infty)$.)


\section{Relationship to other extreme value models}\label{Sec:Relations}

Recall that the classical multivariate extreme value model describes the asymptotic behavior of $(X,Y)$ on sets where at least one component is large, while the CEVM for $(X,Y)$ and for $(Y,X)$ describes the behavior of $(X,Y)$ for large values of $Y$ and of $X$, respectively. Therefore, one might expect that the former implies the latter  pair, and vice versa. However, it turns out that the relationship is much more intricate than this crude reasoning suggests.
The following lemma gives a simple result in this spirit.  Recall Remark \ref{rem:CEV}(i) about the extension of measures.
\begin{proposition}  \label{lem:relat}
  Define the map $S:\R^2\to\R^2$ by $S(x,y):=(y,x)$.
  \begin{enumerate}
    \item Suppose \eqref{eq:margconvX}--\eqref{eq:MEVvague} hold with $\gamma_X,\gamma_Y\le 0$ and a non-degenerate limit measure $\mu$ satisfying $\mu(\bar E^{(\gamma_X)}\times \bar E^{(\gamma_Y)})>0,$ i.e.\ $X$ and $Y$ are asymptotically dependent. Then $(X,Y)$ $\in \CEV(a,b,c,d,$ $\gamma_Y,\mu_{X,Y>})$ and $(Y,X)\in \CEV(c,d,a,b,\gamma_X,\mu_{Y,X>})$ with $\mu_{X,Y>}(\cdot):=$ $\mu\big(\cdot\cap(\bar\R\times \bar E^{(\gamma_Y)})\big)$ and $\mu_{Y,X>}(\cdot):= \mu^S\big(\cdot\cap(\bar\R \times \bar E^{(\gamma_X)})\big)=\mu\big(S^{-1}(\cdot\cap(\bar\R \times \bar E^{(\gamma_X)}))\big)$.

        If $\mu(\{-\infty\}\times\R)=0=\mu(\R\times\{-\infty\})$, then also condition (ii*) is met for both CEVM.
    \item Conversely, if $(X,Y)\in \CEV(a,b,c,d,\gamma_Y,\mu_{X,Y>})$ and $(Y,X)\in \CEV(c,d,a,b,\gamma_X,\mu_{Y,X>})$, then $(X,Y)$ follows a classical extreme value model and \eqref{eq:MEVvague} holds for the non-degenerate limit measure
        \begin{equation} \label{eq:murep1}
        \mu(\cdot):=\mu_{X,Y>}^{pr}\big(\cdot\cap(\bar{\bar E}^{(\gamma_X)}\times {\bar E}^{(\gamma_Y)})\big) + \mu_{Y,X>}^{S\circ pr}\big(\cdot \cap ({\bar E}^{(\gamma_X)}\times\{q_{\gamma_Y}\})\big)
        \end{equation}
        with $pr(x,y):=(x\vee q_{\gamma_X},y\vee q_{\gamma_Y})$ denoting the projection of $ \bar \R^2\setminus([-\infty,q_{\gamma_X}]\times[-\infty,q_{\gamma_Y}])$ onto $([q_{\gamma_X},\infty]\times[q_{\gamma_Y},\infty])\setminus\{(q_{\gamma_X},q_{\gamma_Y})\}$. Hence, $\mu$ is given by
        \begin{eqnarray}
          \mu([q_{\gamma_X},x]\times(y,\infty]) & = & \mu_{X,Y>}([-\infty,x]\times(y,\infty]) \qquad \forall\, x\ge q_{\gamma_X}, y>q_{\gamma_Y}, \label{eq:murep3}\\
          \mu((x,\infty]\times [q_{\gamma_Y},y]) & = & \mu_{Y,X>}([-\infty,y]\times(x,\infty]) \qquad \forall\, x> q_{\gamma_X}, y\ge q_{\gamma_Y}. \label{eq:murep4}
        \end{eqnarray}
         If condition (ii*) holds for both CEVM, then $\mu(\{q_{\gamma_X}\}\times [q_{\gamma_Y},\infty])=0=\mu([q_{\gamma_X},\infty]\times\{q_{\gamma_Y}\})$.
    \end{enumerate}
\end{proposition}
\begin{proof}
   Recall that \eqref{eq:MEVvague} is equivalent to
  \eqref{eq:MEVconvsurv}, which is in turn equivalent to
  \begin{eqnarray}
    t  P\Big\{\frac{X-b(t)}{a(t)}\le x, \frac{Y-d(t)}{c(t)}>y\Big\} \to \mu \big([q_{\gamma_X},x]\times (y,\infty]\big) & & \forall\, x\ge q_{\gamma_X}, y>q_{\gamma_Y}, \label{eq:MEVchar1}\\
    t  P\Big\{\frac{X-b(t)}{a(t)}> x, \frac{Y-d(t)}{c(t)}\le y\Big\} \to \mu \big((x,\infty]\times [q_{\gamma_Y},y]\big) & & \forall\, x> q_{\gamma_X}, y\ge q_{\gamma_Y},\label{eq:MEVchar2}
  \end{eqnarray}
 with suitable extensions of measures as in Remark \ref{rem:CEV}(i).

  {\bf Proof of (i):} We only  prove $(X,Y)\in \CEV(a,b,c,d,\gamma_Y,\mu_{X,Y>})$, as the second assertion follows by completely analogous arguments and the assertion on condition (ii*) is obvious.

  Note that $q_\gamma=-\infty$ for $\gamma\le 0$ so that the maxima in \eqref{eq:MEVconv1} can be omitted. Moreover, by \eqref{eq:margconvY}, $\mu\big(\bar\R\times[q^{\gamma_Y},\infty]\big)=0$, and so
  $\mu_{X,Y>}(\cdot)=\mu\big(\cdot\cap (\bar\R\times(-\infty,\infty])\big)$.

  Now consider an arbitrary Borel set $A\subset\bar\R\times \bar E^{(\gamma_Y)}$ such that $\inf\{y\mid (x,y)\in A\}>q_{\gamma_Y}=-\infty$ and $\mu_{X,Y>}(\partial A)=0$. Then $\mu(\partial A)=\mu_{X,Y>}(\partial A)=0$ and by \eqref{eq:MEVconv1}
  $$ t P\Big\{\Big(\frac{X-b(t)}{a(t)},\frac{Y-d(t)}{c(t)}\Big)\in A\Big\}\to \mu(A)=\mu_{X,Y>}(A),
  $$
  which proves \eqref{eq:CEVconv} and $\mu_{X,Y>}\big(\bar\R\times(y,\infty]\big)<\infty$ for all $y\in\bar E^{(\gamma_Y)}$. Moreover,  \eqref{eq:margconvX} ensures condition (ii) of Definition \ref{def:CEV}.

  Suppose $0=\mu_{X,Y>}\big((\bar \R\setminus\{x\})\times(y,\infty]\big)=\mu\big((\bar \R\setminus\{x\})\times(y,\infty]\big)$ for some $x\in\bar\R$ and $y\in E^{(\gamma_Y)}$. If $\gamma_X<0$ and $x\ge q^{\gamma_X}$, then $\mu\big((\bar\R\setminus\{x\})\times(y,\infty]\big)= \mu(\bar\R\times (y,\infty])=(1+\gamma_Yy)^{-1/\gamma_Y}>0$, contradicting the last assumption. Else, for $\mu^*$ as in \eqref{eq:MEVstandvague},
  $$ 0 = \mu\big((\bar \R\setminus\{x\})\times(y,\infty]\big)=\mu^* \Big( \big( [0,\infty]\setminus\big\{(1+\gamma_X x)^{1/\gamma_X}\big\}\big)\times \big( (1+\gamma_Y y)^{1/\gamma_Y},\infty\big]\Big)=0,
  $$
  and, by the homogeneity of $\mu^*$, even $\mu\big(\bar E^{(\gamma_X)}\times\bar E^{(\gamma_Y)}\big)= \mu^*\big((0,\infty]^2\big)=0$, contradicting the assumptions on $\mu$. Hence, in any case one has $\mu_{X,Y>}\big((\bar \R\setminus\{x\})\times(y,\infty]\big)>0$, which shows that condition (i) of Definition \ref{def:CEV} is fulfilled, too.

 {\bf Proof of (ii):} Standard arguments show that $D_X:=\big\{x\in\bar\R\mid \mu_{X,Y>}(\{x\}\times \bar E^{(\gamma_Y)})>0\big\}$ and $D_Y:=\big\{y\in\bar\R\mid \mu_{Y,X>}(\{y\}\times \bar E^{(\gamma_X)})>0\big\}$ are (at most) countable. For all $x\in E^{(\gamma_X)}\setminus D_X$ and $y\in E^{(\gamma_Y)}\setminus D_Y$ one has
 $$ \mu_{X,Y>}\big((x,\infty]\times(y,\infty]\big) =
 \lim_{t\to\infty} t
 P\Big\{\frac{X-b(t)}{a(t)}>x,\frac{Y-d(t)}{c(t)}>y\Big\} =
 \mu_{Y,X>}\big((y,\infty]\times(x,\infty]\big),
 $$
 which shows that $\mu_{X,Y>}$ and $\mu_{Y,X>}^S$ coincide on $(q_{\gamma_X},\infty]\times (q_{\gamma_Y},\infty]$. It also implies that
 $$\mu_{X,Y>}([q^{\gamma_X},\infty] \times (q_{\gamma_Y},\infty])=0=\mu_{X,Y>}((q_{\gamma_X},\infty] \times [q^{\gamma_Y},\infty]),$$
 and analogous equations for $\mu_{Y,X>}$.
Let now first $x\ge q_{\gamma_X}$ and $y>q_{\gamma_Y}$. Then
\begin{eqnarray*}
t  P\Big\{\frac{X-b(t)}{a(t)}\le x, \frac{Y-d(t)}{c(t)}>y\Big\}
&\to & \mu_{X,Y>}\big([-\infty,x]\times (y,\infty ]\big)\\
 & = & \mu_{X,Y>}\big([-\infty,\min(x,q^{\gamma_X})]\times (y,q^{\gamma_Y} ]\big) \\
&=& \mu_{X,Y>}^{pr}\big([q_{\gamma_X},\min(x,q^{\gamma_X})]\times (y,q^{\gamma_Y} ]\big)\\
& = & \mu([q_{\gamma_X},x] \times (y,\infty]),
\end{eqnarray*}
which shows \eqref{eq:murep3} and \eqref{eq:MEVchar1}. Furthermore, for $x> q_{\gamma_X}$ and $y\geq q_{\gamma_Y}$ we have
\begin{eqnarray*}
\lefteqn{t  P\Big\{\frac{X-b(t)}{a(t)}>x, \frac{Y-d(t)}{c(t)}\leq y\Big\}}\\
&\to & \mu_{Y,X>}\big([-\infty,y]\times (x,\infty ]\big)\\
&=& \mu_{Y,X>}\big([-\infty,q_{\gamma_Y}]\times (x,q^{\gamma_X} ]\big)+\mu_{Y,X>}\big((q_{\gamma_Y},\min(y, q^{\gamma_Y})]\times (x,q^{\gamma_X} ]\big)\\
&=& \mu_{Y,X>}^{S \circ pr}\big((x,q^{\gamma_X} ]\times \{q_{\gamma_Y}\}\big)+ \mu_{X,Y>}^{pr}\big((x,q^{\gamma_X} ]\times (q_{\gamma_Y},\min(y, q^{\gamma_Y})]\big)\\
& = & \mu\big((x,\infty]\times[q_{\gamma_Y},y]\big),
\end{eqnarray*}
which shows \eqref{eq:murep4} and \eqref{eq:MEVchar2}. Moreover, our assumptions imply that \eqref{eq:margconvX} and \eqref{eq:margconvY} hold, and hence $\mu(\{\infty\} \times [q_{\gamma_Y},\infty])=0=\mu([q_{\gamma_Y},\infty] \times \{\infty\})$. Therefore, \eqref{eq:MEVvague} holds.

 Again the assertion about condition (ii*) is trivial. \qed
\end{proof}

The conditions under which we have proved the relation between the classical multivariate extreme value model and the two CEVM are quite restrictive, in that we assume $\gamma_X,\gamma_Y\le 0$ in Proposition~\ref{lem:relat} (i) and we require that the same normalizing functions are used in both CEVM in assertion (ii). We will show in two examples that in general one cannot dispense with these assumptions.

\begin{example} \label{ex:MEVnotimpCEV}
  Let $g(y):= y(2+\sin\log y)$ for $y\ge 1$. Then $g$ has a positive derivative and it is thus strictly increasing and invertible with a strictly increasing inverse $g^\leftarrow:[2,\infty)\to [1,\infty)$. Let $B$ and $Y$ be independent random variables such that $B$ is uniformly distributed on $\{0,1\}$ (i.e.\ it is Bernoulli($1/2$)) and $Y$ is standard Pareto distributed. Define $X:=BY+(1-B)(-g^\leftarrow(2Y))$. Then $Y$ fulfills \eqref{eq:margconvY} with $c(t)=d(t)=t$ and $\gamma_Y=1$, and $X$ satisfies \eqref{eq:margconvX} for $a(t)=b(t)=t/2$ and $\gamma_X=1$, because
  $ tP\{(X-b(t))/a(t)>x\}= tP\{B=1, Y>(t/2)(x+1)\}=(1+x)^{-1}$ for $x>-1$ and $t$ sufficiently large.

  Moreover, convergence \eqref{eq:MEVvague} holds, because for all $x,y> -1$ and sufficiently large $t$
  \begin{eqnarray*}
    \lefteqn{tP\Big\{\max\Big(\frac{X-b(t)}{a(t)},-1\Big)>x \text{ or } \max\Big(\frac{Y-d(t)}{c(t)},-1\Big)>y\Big\}} \\
    & = & tP\Big\{\frac{X-b(t)}{a(t)}>x \Big\} + tP\Big\{\frac{Y-d(t)}{c(t)}>y\Big\} - tP\Big\{B=1,Y>\frac t2(1+x), Y>t(1+y)\Big\} \\
    & = & (1+x)^{-1}+(1+y)^{-1} - \min\big((1+x)^{-1},(2 (1+y))^{-1}\big)\\
    & = & (2(1+y))^{-1} + \max\big((1+x)^{-1},(2 (1+y))^{-1}\big)\\
    & = & \mu\big([-1,\infty]^2\setminus ([-1,x]\times[-1,y])\big)
  \end{eqnarray*}
  where $\mu$ denotes the measure concentrated on $\{(t,(t-1)/2)|t>-1\}\cup\{(-1,t)|t>-1\}$ such that $\mu\big(\{(t,(t-1)/2)|t>r\}\big)=1/(1+r)$ and $\mu\big(\{(-1,t)|t>r\}\big)=1/(2(1+r))$ for all $r>-1$.

  However, $(X,Y)$ does not fulfill the conditions of any  CEVM incl.\ (ii*). Suppose $(X,Y)\in$ \linebreak $ \CEV^*(\alpha,\beta,t,t,1,\mu_{X,Y>})$ for some normalizing functions $\alpha>0$ and $\beta\in\R$. Recall that \cite{HR07} have shown that then the normalizing functions satisfy the relations
  \eqref{eq:normregvar} where without loss of generality one may assume $C\ne 0$.
  (Note that one needs condition (ii*) in order to apply the convergence to types theorem as it is done in that paper.) In particular, $\alpha$ is regularly varying with index $\rho$, $|\beta|$ is regularly varying with index $\rho$ or 0, and, by Theorem B.2.2 of \cite{dHF06}, $\beta(t)/\alpha(t)\to C/\rho$ if $\rho>0$.

  Let $g(y):= 2$ for $y<1$. For all $y>-1$ and sufficiently large $t$, we have
  \begin{eqnarray}
    \lefteqn{ t P\Big\{\frac{X-\beta(t)}{\alpha(t)}\le x, \frac{Y-t}t>y\Big\} } \nonumber\\
    & = & \frac t2 P\big\{Y\le \beta(t)+\alpha(t)x, Y>t(1+y)\big\} +
    \frac t2 P\big\{g^\leftarrow(2Y)\ge -(\beta(t)+\alpha(t)x),Y>t(1+y)\big\}\nonumber\\
    & = & \frac 12 \Big((1+y)^{-1}-\frac t{\beta(t)+\alpha(t)x}\Big)^+ 1_{(1,\infty)}(\beta(t)+\alpha(t)x) \nonumber\\
    & & 
    +\min\Big((2(1+y))^{-1}, \frac t{g\big(-(\beta(t)+\alpha(t)x)\big)}\Big).
      \label{eq:CEVex1}
  \end{eqnarray}

  First assume  $\alpha(t_n)=o(t_n)$ for some sequence $t_n\to\infty$. Then by the above remarks also $\beta(t_n)=o(t_n)$, and the first term of \eqref{eq:CEVex1} vanishes asymptotically for all $x\in\R$. Together with \eqref{eq:margconvY}, this implies $\mu(\{\infty\}\times (y,\infty])\ge (2(1+y))^{-1}$, in contradiction to condition (ii) of Definition \ref{def:CEV}.

  Hence $t=O(\alpha(t))$, which implies $\rho\ge 1$ and thus $\beta(t)/\alpha(t)\to C/\rho$ and $\beta(t)+\alpha(t)x=\alpha(t)(x+C/\rho +o(1))$.

  Now, for $x>-C/\rho$ and sufficiently large $t$, \eqref{eq:CEVex1} equals
  $\big((1+y)^{-1}-t/(\alpha(t)(x+C/\rho+o(1)))\big)^+/2+(2(1+y))^{-1}$. For all $y>-1$, the first summand of this expression has to converge to a positive limit for some $x>-C/\rho$, since otherwise \eqref{eq:CEVconv} or (ii) of Definition \ref{def:CEV} does not hold.  Therefore, $t/\alpha(t)\rightarrow \alpha_0 \in [0,\infty)$.

  Consider first the case $\alpha_0=0$, i.e.\ $t=o(\alpha(t))$. Since $g(z)\in [z,3z]$ for all $z\ge 1$, $t/g\big(-(\beta(t)+\alpha(t)x)\big)$ tend to $0$ for $x<-C/\rho$, and so does \eqref{eq:CEVex1}. On the other hand, if $x>-C/\rho$, then by similar arguments one sees that \eqref{eq:CEVex1} tends to $(1+y)^{-1}=\mu_{X,Y>}(\bar\R\times(y,\infty])$. We conclude that $\mu_{X,Y>}((\R\setminus\{-C/\rho\})\times (y,\infty])=0$, contradicting condition (i) of Definition \ref{def:CEV}.

  Therefore, $\alpha(t)/t$  tends to some $\alpha_0\in (0,\infty)$. Then, however, for $x<-C/\rho$, \eqref{eq:CEVex1}  equals
  $$ \min\Big((2(1+y))^{-1},\frac t{g(-\alpha_0 t(x+C/\rho)(1+o(1)))}\Big). $$
  Direct calculations show that along the sequence $t_n=-\exp(2\pi n)/(\alpha_0(x+C/\rho))$ this expression converges to $\min\big( (2(1+y))^{-1},(-2\alpha_0(x+C/\rho))^{-1}\big)$, while along the sequence
   $\tilde t_n=-\exp(2\pi n+\pi/2)/(\alpha_0(x+C/\rho))$ it converges to $\min\big( (2(1+y))^{-1},(-3\alpha_0(x+C/\rho))^{-1}\big)$.

   We have thus proved that \eqref{eq:CEVex1} cannot converge for any choice of $\alpha(t)$ and $\beta(t)$ on a dense set of $x$-values to a non-degenerate limit, that is, $(X,Y)$ does not satisfy the conditions of a CEVM incl.\ (ii*). \qed
\end{example}

A closer inspection of Example \ref{ex:MEVnotimpCEV} reveals that the problem arises from an irregular behavior of $(X-b(t))/a(t)$ on $(-\infty,-1)=(-\infty,q_{\gamma_X})$ for large values of $Y$. Convergence \eqref{eq:MEVvague} mainly describes the behavior of $\big((X-b(t))/a(t),(Y-d(t))/c(t)\big)$ on $([q_{\gamma_X},\infty]\times [q_{\gamma_Y},\infty])\setminus\{(q_{\gamma_X},q_{\gamma_Y})\}$, while on $(-\infty,q_{\gamma_X})$  essentially only the total mass of $(X-b(t))/a(t)$ is known (for large values of $Y$). Therefore, an irregular behavior as in Example \ref{ex:MEVnotimpCEV} can only be ruled out in the case $q_{\gamma_X}=-\infty$.
\begin{remark}
\begin{enumerate}
\item  In Example \ref{ex:MEVnotimpCEV}, the assumptions of the CEVM fail only because of the behavior of $(X-b(t))/a(t)$ below $q_{\gamma_X}$. Hence one may establish a limit for
  $$  t P\Big\{ \max\Big(\frac{X-b(t)}{a(t)},q_{\gamma_X}\Big)\le x, \frac{Y-d(t)}{c(t)}>y\Big\}. $$
  However, this convergence cannot readily be interpreted in terms of a CEVM for a modified vector $(\tilde X,Y)$.
  \item If $\gamma_X, \gamma_Y>0$ and $X,Y\geq 0$, the assumptions \eqref{eq:margconvX}--\eqref{eq:MEVvague} imply that the vector $(X^{1/\gamma_X},Y^{1/\gamma_Y})$ is standard regularly varying on $[0,\infty]^2 \setminus\{(0,0)\}$ with some limit measure $\nu$. If $\nu$ is not concentrated on the axes, i.e.\ if there is asymptotic dependence between the two components, then $(X^{1/\gamma_X},Y^{1/\gamma_Y})$ also satisfies a CEVM and so does $(X,Y)$, cf.\ \cite{HR07}, Section 5. Similarly, one can also handle the case when $X$ and $Y$ are not necessarily non-negative but they have a finite lower bound.
\end{enumerate}
\end{remark}

Concerning the converse implication, the following example shows that assertion (ii) of Proposition~\ref{lem:relat} does not hold if one does not require quite strong restrictions on the relation between the normalizing functions for the conditional extreme value models for $(X,Y)$ and $(Y,X)$. In particular, it follows that  Theorem 2.1 of \cite{DR11a} is not correct in its present form.

\begin{example} \label{ex:CEVnotimpMEV}
  For $ 0<u\le 1$, let $g_c(u)=u(1+c\sin\log u)$. Direct calculations show that $g_c$ is strictly increasing for all $|c|<1/\sqrt{2}$ with $g_c(1)=1$ and $\lim_{u\downarrow 0} g_c(u)=0$. Denote its inverse by $g_c^\leftarrow:(0,1]\to (0,1]$, and define a strictly decreasing function by $\psi_c(z)= g_c^\leftarrow(1/z)$, $z\ge 1$. Let $Z$ be a standard Pareto random variable and $B$ an independent discrete random variable that is uniformly distributed on $\{1,\ldots,4\}$. Define
  $$ (X,Y) := \left\{
     \begin{array}{l@{\quad}l}
        \big(2-Z^{-1}, 2-\psi_{1/2}(Z)\big), & B=1,\\
        \big(2-Z^{-1/2}, 2-\psi_{-1/2}(Z)\big), & B=2,\\
        \big(2-Z^{-1}, 1-Z^{-1}\big), & B=3,\\
        \big(1-Z^{-1}, 2-Z^{-1}\big), & B=4.
     \end{array}
     \right.
  $$
  Obviously, $X,Y\in (0,2)$ almost surely.
  \smallskip

  {\bf Assertion 1}: $(X,Y)\in \CEV^*(1,0,c,d,-1,\mu_{X,Y>})$ with $c(t)=4/(3t), d(t)=2-4/(3t)$ and the measure $\mu_{X,Y>}= \big(\frac 13\delta_1+\frac 23\delta_2\big)\otimes \nu_{-1}$, where $\nu_{-1}$ is given by $\nu_{-1}((r,\infty])=1-r$ for all $r\le 1$ (i.e.\ $\nu_{-1}$ is the Lebesgue measure restricted to $(-\infty,1]$).\\[0.5ex]
  For all $x\in\R, y\le 1$ and sufficiently large $t$, one has
  \begin{eqnarray*}
    \lefteqn{t P\Big\{X\le x, \frac{Y-d(t)}{c(t)}>y\Big\}}\\
    & = & \frac t4 \bigg(P\Big\{ 2-Z^{-1}\le x, \psi_{1/2}(Z)<\frac 4{3t}(1-y)\Big\} + P\Big\{ 2-Z^{-1/2}\le x, \psi_{-1/2}(Z)<\frac 4{3t}(1-y)\Big\} \\
    & & \hspace*{1cm} + P\Big\{ 1-Z^{-1}\le x, Z^{-1}<\frac 4{3t}(1-y)\Big\}\bigg) \\
    & = & \frac t4 \bigg(P\Big\{  Z>\frac 1{ g_{1/2}(4(1-y)/(3t))}\Big\} + P\Big\{ Z>\frac 1{ g_{-1/2}(4(1-y)/(3t))}\Big\}\bigg)1_{[2,\infty)}(x) \\
    & & \hspace*{1cm} + \frac t4 P\Big\{  Z>\frac{3t}{ 4(1-y)}\Big\}1_{[1,\infty)}(x) \\
    & = & \frac t4\big(g_{1/2}(4(1-y)/(3t))+g_{-1/2}(4(1-y)/(3t))\big) 1_{[2,\infty)}(x)+ \frac{1-y}31_{[1,\infty)}(x)\\
    & = & \frac{1-y}31_{[1,\infty)}(x)+\frac{2(1-y)}31_{[2,\infty)}(x)\\
    & = & \mu_{X,Y>}((-\infty,x]\times(y,\infty]),
  \end{eqnarray*}
  which proves \eqref{eq:CEVconv}.
  In particular, \eqref{eq:margconvY} holds with $\gamma_Y=-1$. Note that the limit measure is a product measure and that the marginal measure corresponding to $X$ has mass concentrated at the points $1$ and $2$; hence the non-degeneracy conditions are fulfilled.
  \smallskip

  {\bf Assertion 2}: $(Y,X)\in \CEV^*(1,0,a,b,-1,\mu_{Y,X>})$ with $a(t)=2/t, b(t)=2-2/t$ and $\mu_{Y,X>}=\nu_{-1}\otimes (\frac 12\delta_1+\frac 12\delta_2)$.\\[0.5ex]
  This assertion follows by similar arguments from
  \begin{eqnarray*}
    \lefteqn{t P\Big\{Y\le y, \frac{X-b(t)}{a(t)}>x\Big\}}\\
    & = & \frac t4 \bigg(P\Big\{ 2-\psi_{1/2}(Z)\le y, Z^{-1}<\frac 2t(1-x)\Big\} + P\Big\{ 2-\psi_{-1/2}(Z)\le y, Z^{-1/2}<\frac 2t(1-x)\Big\}\\
    & & \hspace*{1cm}  + P\Big\{ 1-Z^{-1}\le y, Z^{-1}<\frac 2t(1-x)\Big\} \bigg) \\
    & \to & \frac{1-x}2 1_{[2,\infty)}(y)+ \frac{1-x}2 1_{[1,\infty)}(y)\\
    & = & \mu_{Y,X>}((-\infty,y]\times(x,\infty]),
  \end{eqnarray*}
  for all $x\le 1$ and $y\in\R$. In particular, \eqref{eq:margconvX} holds with $\gamma_X=-1$.

  \smallskip

  {\bf Assertion 3}: Convergence \eqref{eq:MEVvague} does not hold true.\\[0.5ex]
  If \eqref{eq:MEVvague} holds, then because of \eqref{eq:margconvX} and \eqref{eq:margconvY} also the following expression must converge for all $x,y<1$:
  \begin{eqnarray*}
    \lefteqn{t P\Big\{\frac{X-b(t)}{a(t)}>x, \frac{Y-d(t)}{c(t)}>y\Big\}}\\
    & = & \frac t4 \bigg( P\Big\{ Z>\frac t{2(1-x)}, \psi_{1/2}(Z)<\frac 4{3t}(1-y)\Big\} + P\Big\{ Z>\Big(\frac t{2(1-x)}\Big)^2, \psi_{-1/2}(Z)<\frac 4{3t}(1-y)\Big\}\bigg) \\
    & = & \frac t4 P\Big\{ Z>\max\Big(\frac t{2(1-x)},\frac 1{ g_{1/2}(4(1-y)/(3t))}\Big)\Big\} + o(1)
    \\
    & = & \frac t4 \min \Big(\frac {2(1-x)}t, g_{1/2}(4(1-y)/(3t))\Big).
  \end{eqnarray*}
  However, it is easily seen that the limit inferior of the last expression equals $\min\big((1-x)/2,$ \linebreak[1] $(1-y)/6\big)$, whereas its limit superior is $\min\big((1-x)/2,(1-y)/2\big)$. (Choose $t=4(1-y)\exp((2n+1/2)\pi)/3$ and $t=4(1-y)\exp((2n-1/2)\pi)/3$, respectively.) \qed
\end{example}
Note that in this example the conditional extreme  value models convey rather limited information on the non-extreme component, in that they only specify the two points in the neighborhood of which this component will lie for large values of the other component. This allows for a very irregular behavior of the non-extreme component, say $Y$, on a finer scale (see case $B=1$). If this irregular behavior for large values of $X$ is balanced by the behavior for smaller values of $X$ (see case $B=2$), then the marginal tail behavior of $Y$ can be regular. In the classical multivariate extreme value model, the behavior of $Y$ for large values of $X$ is examined in much more detail in the vicinity of the higher point of mass, so that the irregular behavior rules out the convergence required by this model.

To avoid such effects, one may require that the normalizations in both CEVM  are identical, or, formally weaker, that they are equivalent in the sense of the convergence to types theorem, i.e., $(X,Y)\in \CEV^*(\alpha,\beta,c,d,\gamma_Y,\mu_{X,Y>})$ and $(Y,X)\in \CEV^*(\chi,\delta,a,b,\gamma_X,\mu_{Y,X>})$ where
\begin{eqnarray}
  \frac{\alpha(t)}{a(t)} \to A\in (0,\infty), & & \quad \frac{\beta(t)-b(t)}{a(t)} \to B\in\R  \label{eq:normequiv}\\
  \frac{\chi(t)}{c(t)} \to C\in (0,\infty), & & \quad \frac{\delta(t)-d(t)}{c(t)} \to D\in\R \nonumber
\end{eqnarray}
\cite{DR11a} used in Proposition 4.1 somewhat weaker conditions on the nomalizing functions when they concluded convergence \eqref{eq:MEVvague} from \eqref{eq:margconvX} and $(X,Y)\in \CEV(\alpha,\beta,c,d,$ $\gamma_Y,\mu_{X,Y>})$, provided $\lim_{t\to\infty}\alpha(t)/a(t)= A<\infty$ and $\beta(t)$ and $b(t)$ converge to the same limit in $(-\infty,\infty]$. The following example shows that in some cases these conditions are too weak to ensure that the conclusion holds, because they do not rule out that the normalizing functions $\beta$ and $b$ behave quite differently in the Gumbel case $\gamma_X=0$.
\begin{example} \label{ex:CEVmargnotimplyMEV}
  For $x\ge \e$, let $\psi(x):=\log x +\sin\log\log x$, which is a continuous and increasing function with increasing inverse $\psi^\leftarrow:[1,\infty)\to [\e,\infty)$. Observe that,  for a sufficiently large $x_0>1$, the function $g(x):= 4/(3\psi^\leftarrow(x))-\e^{-x}/3$ is decreasing on $[x_0,\infty)$ with values in $(0,1]$, because
  $$ \frac d{dx}g(\psi(x))= -\frac 4{3x^2} + \frac 1{3x^2}\exp(-\sin\log\log x)\Big(1+\frac{\cos\log\log x}{\log x}\Big) \le \frac 1{3x^2}\Big(-4+\e \frac 43\Big)<0
  $$
  for $x\ge \e^3$.

  Let $Z$ be a random variable with survival function $P\{Z>x\}=g(x)1_{[x_0,\infty)}(x)+1_{(-\infty,x_0)}(x)$. Let $Y$ be a standard Pareto random variable independent of $Z$ and define
  $$ X:= \left\{
    \begin{array}{l@{\quad}l}
      \log (Y/4), & Y>4,\\
      Z, & Y\le 4.
    \end{array}
    \right.
  $$
  Obviously, \eqref{eq:margconvY} holds with $c(t)=d(t)=t$ and $\gamma_Y=1$.

  Moreover, \eqref{eq:margconvX} is satisfied with $b=\psi$ and $a(t)=1$. To see this, check that for $x>x_0$
  $$ P\{X>x\} = P\{\log(Y/4)>x,Y>4\}+P\{Z>x\}\cdot P\{Y\le 4\} = \frac 14\e^{-x} + \frac 34 g(x) = \frac 1{\psi^\leftarrow (x)},
  $$
  which implies that the quantile function $F_X^\leftarrow$ of $X$ (i.e.\ the generalized inverse of the cdf $F_X$) is given by $U_X(t) := F_X^\leftarrow(1-1/t)=\psi(t)$ for sufficiently large $t$.
  Hence
  \begin{eqnarray*}
    U_X(tx)-U_X(t) & = & \log x + \sin \log(\log t+\log x)-\sin\log\log t\\
     & = & \log x + \sin (\log\log t+o(1))-\sin\log\log t \\
     & \to & \log x
  \end{eqnarray*}
  by the uniform continuity of the sine function. This convergence is known to imply \eqref{eq:margconvX} with  $\gamma_X=0,$ $b=U_X=\psi$ and $a(t)=1$ (see, e.g., \cite{dHF06}, Theorem 1.1.2).

  For $\alpha(t)=1$ and $\beta(t)=\log t$, $x\in\R$ and $y>-1$, we have eventually
  \begin{eqnarray*}
    tP\Big\{\frac{X-\beta(t)}{\alpha(t)}>x, \frac{Y-d(t)}{c(t)}>y\Big\}
    & = & tP\big\{\log(Y/4)>x+\log t, Y>t(1+y)\big\}\\
    & = &tP\big\{ Y>\max(4\e^x,1+y)t\big\} \\
    & = & \big(\max(4\e^x,1+y)\big)^{-1} \\
    & = & \mu_{X,Y>}((x,\infty]\times (y,\infty]),
  \end{eqnarray*}
  where the measure $\mu_{X,Y>}$ is concentrated on the set $\{(\log(t/4),t-1)|t>0\}$ with $\mu_{X,Y>}\{(\log(t/4),t-1)|t>r\}=r^{-1}$.
  This shows $(X,Y)\in \CEV^*(\alpha,\beta,c,d,1,\mu_{X,Y>})$.  Note that $\alpha=a$ and that both $\beta(t)$ and $b(t)$ tend to $\infty$, but they are not equivalent in the sense of \eqref{eq:normequiv}.

  Finally, we show that, in contrast to what is claimed in Proposition 4.1 of \cite{DR11a}, convergence \eqref{eq:MEVvague} does not hold for the above choices of the normalizing functions. In view of \eqref{eq:margconvX} and \eqref{eq:margconvY}, it suffices to show that for some $x\in\R$ and $y>-1$
  \begin{eqnarray*}
  tP\Big\{ \frac{X-b(t)}{a(t)}>x,\frac{Y-d(t)}{c(t)}>y\Big\}
  & = &        tP\big\{\log(Y/4)>x+\psi(t), Y>t(1+y)\big\}\\
  & = & \big(\max\big(4\exp(x+\sin\log\log t),1+y\big)\big)^{-1}
  \end{eqnarray*}
  does not converge. This, however, is obvious if one considers the sequences $t_n=\exp(\exp(2\pi n))$ and $\tilde t_n=\exp(\exp(2\pi n+\pi/2))$. \qed
\end{example}

 We conclude this section with an example about the connection between a CEVM and so-called hidden regular variation. Here we consider  a random vector $(X,Y) \in [0,\infty)^2$ that is standard regularly varying on $[0,\infty]^2 \setminus \{(0,0)\}$, i.e.\
\begin{equation}\label{Eq:HRV1} t P\Big\{ \Big( \frac{X}{t}, \frac{Y}t\Big) \in \cdot \Big\} \,\vto\, \mu^* (\cdot)
\end{equation}
vaguely in $[0,\infty]^2 \setminus \{(0,0)\}$ with a non-degenerate limit measure $\mu^*$ satisfying $\mu^*(\{\infty\}\times[0,\infty])=0=\mu^*([0,\infty]\times\{\infty\})$; cf.\ \eqref{eq:MEVstandvague}. The random vector is called hidden regularly varying if, in addition, for a normalizing function $\lambda_0$ with $t=o(\lambda_0(t))$ and a non-degenerate limit measure $\mu_0$ on $(0,\infty]^2$ with $\mu_0(\{\infty\}\times (0,\infty])=0=\mu_0((0,\infty]\times \{\infty\})$,
\begin{equation}\label{Eq:HRV2} \lambda_0(t) P\Big\{ \Big( \frac{X}{t}, \frac{Y}t\Big) \in \cdot \Big\} \,\vto\, \mu_0 (\cdot)
\end{equation}
vaguely in $(0,\infty]^2$, i.e.
$$ \lambda_0(t) P\Big\{ \Big( \frac{X}{t}, \frac{Y}t\Big) \in B \Big\} \to \mu_0 (B)
$$
 for all $\mu_0$-continuous Borel sets  $B \subset (0,\infty)^2$ which are bounded away from both axes. Hidden regular variation is only possible for a vector with asymptotically independent components, i.e.\ in the case $\mu^*((0,\infty)^2)=0$, since else the normalizing function $\lambda_0(t)$ is of the order $t$.

 \cite{DR11a} conjectured in Section 5 that $(X,Y)$ is hidden regularly varying if it fulfills \eqref{Eq:HRV1} with $\mu^*((0,\infty)^2)=0$ and  the conditions of a CEVM as well. The following counterexample shows that in general this implication does not hold.
\begin{example}
Let $Z_1, Z_2$ and $B$ be independent random variables, $Z_1$ be standard Pareto and $B$  uniformly distributed on $\{1,2,3\}$. Moreover,  $P(Z_2>x)=x^{-2}(2+\sin\log(x))/2$ for all $ x \geq 1$, which is easily seen to be a decreasing function. For some $\tau\in (0,1/2)$, define
$$ (X,Y):= \left\{
  \begin{array}{l@{\quad}l@{\quad}l}
(Z_1^\tau,Z_1) & & B=1,\\
(Z_1,Z_1^\tau) & \text{if} & B=2,\\
(Z_2,Z_2) & & B=3.
  \end{array}
  \right.
$$
Let $c(t)=d(t)=t/3$. Then, for all $y>-1$,
\begin{eqnarray}
 tP\Big\{\frac{Y-d(t)}{c(t)}>y\Big\}
&=& \frac{t}{3}P\{Z_1>t(1+y)/3\} + \frac{t}{3}P\{Z_1^\tau>t(1+y)/3\}+\frac{t}{3}P\{Z_2>t(1+y)/3\} \nonumber\\
&\rightarrow& (1+y)^{-1} \label{eq:margconvYex}
\end{eqnarray}
as $t\to\infty$, since both $P\{Z_1^\tau>t(1+y)/3\}=o(t^{-1})$ and $P\{Z_2>t(1+y)/3\}=o(t^{-1})$. Therefore, \eqref{eq:margconvY} holds with $\gamma_Y=1$.

Furthermore, with $\alpha(t)=\beta(t)=(t/3)^\tau$, we have for all $x \in \mathbb{R}, y>-1$
\begin{eqnarray*}
&& tP\Big\{\frac{X-\beta(t)}{\alpha(t)}\leq x,\frac{Y-d(t)}{c(t)}>y\Big\}\\
&=& tP\Big\{\frac{Y-d(t)}{c(t)}>y\Big\}-\frac{t}{3}P\Big\{Z_1^\tau>\Big(\frac t3\Big)^\tau(1+x), Z_1>\frac t3(1+y)\Big\} \\
&& - \, \frac{t}{3}P\Big\{Z_1>\Big(\frac t3\Big)^\tau(1+x),Z_1^\tau>\frac t3(1+y)\Big\}-\frac{t}{3}
P\Big\{Z_2>\Big(\frac t3\Big)^\tau(1+x),Z_2>\frac t3(1+y)\}\\
&\rightarrow& (1+y)^{-1}-\min\big(((1+x)^+)^{-1/\tau},(1+y)^{-1})\\
& = & \mu_{X,Y>}((-\infty,x] \times (y,\infty])
\end{eqnarray*}
as $t\to\infty$. Here $\mu_{X,Y>}$ denotes the measure that is concentrated on $\{(t^\tau-1,t-1)|t>0\}$ such that $\mu_{X,Y>}\{(t^\tau-1,t-1)|t>r\}=r^{-1}$ for all $r>0$. Hence $(X,Y)\in \CEV^*(\alpha,\beta,c,d,1,\mu_{X,Y>})$. By symmetry of the construction, it follows immediately that \eqref{eq:margconvX} holds and that $(Y,X)$ follows a CEVM as well.

Next, note that for $x,y > 0$
\begin{eqnarray*}
 tP\{X>tx \text{ or } Y>ty\}
&=& tP\{X>tx\}+tP\{Y>ty\} -\frac{t}{3}P\{Z_1^\tau>tx, Z_1>ty\}\\
&& -\frac{t}{3}P\{Z_1>tx, Z_1^\tau>ty\}-\frac{t}{3}P\{Z_2>t\max(x,y)\} \\
&\to &     x^{-1}/3+y^{-1}/3\\
&=:&\mu^*\big([0,\infty]^2\setminus([0,x]\times[0,y])\big),
\end{eqnarray*}
i.e.\ \eqref{Eq:HRV1} holds.  Moreover, according to \eqref{eq:margconvYex}, $\mu^*([0,\infty]\times(y,\infty])=\lim_{t\to\infty} P\{Y>ty\}=y^{-1}/3$, and likewise $\mu^*((x,\infty]\times[0,\infty])=\lim_{t\to\infty} tP\{X>tx\}=x^{-1}/3$. Hence,
$$ \mu^*((x,\infty]^2)= \mu^*((x,\infty]\times[0,\infty])+ \mu^*([0,\infty]\times(x,\infty])-\mu^*\big([0,\infty]^2\setminus[0,x]^2\big)=0
$$
for all $x>0$. Thus $\mu^*((0,\infty]^2)=0$, too, i.e.\ $\mu^*$ is concentrated on the axes with $\mu^*((r,\infty)\times\{0\})=\mu^*(\{0\}\times(r,\infty))= (3r)^{-1}$ for all $r>0$.

However, the vector $(X,Y)$ is not hidden regularly varying. To show this, suppose that there exists a function $\lambda_0$ such that \eqref{Eq:HRV2} holds for a non-degenerate limit measure $\mu_0$. Then there exists $x>0$ such that $\mu_0((x,\infty)^2)>0$ and $(x,\infty)^2$ is a $\mu_0$-continuity set, because $\mu_0$ is homogeneous (see \cite{Res07}, Section 9.4.1). Therefore,
\begin{eqnarray*}
 \lefteqn{\lambda_0(t)P\{X>tx, Y>tx\}}\\
&=& \frac{\lambda_0(t)}{3}\left[P\{Z_1^\tau>tx, Z_1>tx\} + P\{Z_1>tx,Z_1^\tau>tx\}+P\{Z_2>tx\}\right]\\
&=& \frac{\lambda_0(t)}{3}\left[2(tx)^{-1/\tau}+(tx)^{-2}(2+\sin\log(tx))/2\right]\\
& = & \frac{\lambda_0(t)}{3}(tx)^{-2}(2+\sin\log(tx))(1+o(1))/2\\
& \to & \mu_0((x,\infty]^2),
\end{eqnarray*}
which implies that $\lambda_0(t) t^{-2}(2+\sin\log(tx))$ converges to a finite positive constant (depending on $x$). However, convergence \eqref{Eq:HRV2} implies that $\lambda_0$ is a regularly varying function; cf. \cite{Res07}, Section 9.4.1. Because $t^{-2}(2+\sin\log(tx))$ is not a regularly varying function of $t$, this leads to a contradiction. \qed
\end{example}

In this example, the extremal behavior of the vector $(X,Y)$ is dominated by the regularly varying random variable $Z_1$ in the following sense: if at least one component of $(X,Y)$ is large, then this is typically due to the fact that $Z_1$ is large. However, if we are interested in extremal events when both components of $(X,Y)$ exceed the same threshold (or, more general, thresholds of the same order of magnitude), then this happens typically because the  random variable $Z_2$ is large, which is not regularly varying. Therefore, the assumptions of hidden regular variation are not met.

\section{Conclusion}
 \label{sect:discussion}

The discussions of the preceding sections show that the CEVM exhibits specific features which are not shared by other extreme value models. These aspects have not been appropriately addressed in the previous literature. (For instance, \cite{DR11b} state that ``The CEV model primarily differs from the multivariate extreme value model in the domain of attraction condition.'') This may be due to the fact that in the standard case where $(X,Y)$ attains values in $[0,\infty)^2$ and both components are divided by the same factor, CEVM, classical extreme value models and the additional convergence \eqref{Eq:HRV2} considered in models with hidden regular variation can all be considered special cases of a general concept of regular variation on cones; see e.g.\ \cite{DMR13}.

However, as the above examples demonstrate, even for vectors $(X,Y)$ that satisfy both the assumptions of a CEVM and of a model from classical MEVT, in general these models focus on very different information about the behavior of $X$ for large values of $Y$. This is partly due to the fact that, after a linear normalization,  the classical model does not discriminate values of $X$ below $q_{\gamma_X}$, whereas in the limit a CEVM may retain more information about the not necessarily extreme component in that region. A more fundamental difference, though, is that  in classical bivariate extreme value models the normalizing functions for $X$ are completely determined by the marginal tail behavior, whereas in a CEVM they must be adapted to the overall behavior of $X$ for large values of $Y$, which need not correspond to large values of $X$ (or $-X$).
While this allows for a more flexible relationship between $X$ and $Y$, in some situations all normalizing functions that meet the conditions of Definition \ref{def:CEV} (including (ii*)) result in a very coarse description of the behavior of $X$.

This problem is most easily illustrated by the example of a mixture. Suppose two random vectors $(X_1,Y)$ and $(X_2,Y)$ with the same second component both fulfill a CEVM. If we first pick one of the vectors at random and consider the conditional extreme value behavior of the resulting mixture model $P\{(X,Y)\in\cdot\}=(P\{(X_1,Y)\in\cdot\}+P\{(X_2,Y)\in\cdot\})/2$, then in some cases the approximating model specified by the limit measure $\mu_{X,Y>}$ and the normalizing functions  neither retain the information encoded in the CEVM for $(X_1,Y)$ nor the information given by the CEVM for $(X_2,Y)$.
\begin{example} \label{ex:mixt}
  Let  $Y$ be regularly varying with index $-1/\gamma_Y$, i.e.\ \eqref{eq:margconvY} holds with $\gamma_Y>0$, and let $X_i=\omega_i-g_i(Y)$, $i\in\{1,2\}$, for some $\omega_1<\omega_2$ and some invertible functions $g_i>0$ which are regularly varying with negative index $-\tau_i$. Then the inverse $g_i^\leftarrow$ is regularly varying at 0 with index $-1/\tau_i$ and with $U_Y(t):=F_Y^\leftarrow(1-1/t)$ one has $g_i^\leftarrow\big(-g_i(U_Y(t))x\big)=|x|^{-1/\tau_i} U_Y(t)(1+o(1))$ for $x\le 0$ . Thus, for  $y>q_{\gamma_Y}=-1/\gamma_Y$ and $x \leq 0$,
  \begin{eqnarray*}
    \lefteqn{tP\Big\{\frac{X_i-\omega_i}{g_i(U_Y(t))}>x, \frac{Y-U_Y(t)}{\gamma_Y U_Y(t)}>y\Big\}} \\
     & = & tP\big\{g_i(Y)<-g_i(U_Y(t))x, Y>U_Y(t)(1+\gamma_Yy)\big\} \\
     & = & t P\big\{ Y>\max\big(g_i^\leftarrow(-g_i(U_Y(t))x), U_Y(t)(1+\gamma_Yy)\big)\big\}\\
     & = & t \min\big(1-F_Y(|x|^{-1/\tau_i} U_Y(t)(1+o(1))), 1-F_Y(U_Y(t)(1+\gamma_Yy))\big)\\
     & \to & \min\big(|x|^{1/(\tau_i\gamma_Y)}, (1+\gamma_Y y)^{-1/\gamma_Y}\big)\\
     & = & \mu_{X,Y>}((x,\infty]\times(y,\infty]),
  \end{eqnarray*}
as $t \to \infty$. Here $\mu_{X,Y>}$ denotes the measure that is concentrated on $\{(-t^{-\tau_i\gamma_Y},(t^{\gamma_Y}-1)/\gamma_Y)|t>0\}$ such that $\mu_{X,Y>}\{(-t^{-\tau_i\gamma_Y},(t^{\gamma_Y}-1)/\gamma_Y)|t>r\}=r^{-1}$ for all $r>0$. Therefore,
  $(X_i,Y)\in \CEV^*(g_i\circ U_Y,\omega_i,\gamma_YU_Y, U_Y,\gamma_Y,\mu_{X,Y>})$.

  The mixture also follows a CEVM, because
  \begin{eqnarray*}
   tP\Big\{X\le x, \frac{Y-U_Y(t)}{\gamma_Y U_Y(t)}>y\Big\}
   & = & \frac t2 \sum_{i=1}^2 P\big\{g_i(Y)\ge \omega_i-x,Y>U_Y(t)(1+\gamma_Y y)\big\} \\
   & \to & \frac 12 \sum_{i=1}^2 1_{[\omega_i,\infty)}(x) (1+\gamma_Y y)^{-1/\gamma_Y}
  \end{eqnarray*}
  for all $x\in\R$ and $y>q_{\gamma_Y}$. However, only $\omega_1$ and $\omega_2$ can be recovered from this CEVM, whereas the more detailed behavior of $X$ near these points for large values of $Y$, which is specified by the CEVM for $(X_1,Y)$ and $(X_2,Y)$, is lost.

  In contrast, the classical bivariate extreme value model for the mixture $P\{(X,Y)\in\cdot\}$ keeps all the information from the model for the component $P\{(X_2,Y)\in\cdot\}$ with the larger point of accumulation $\omega_2>\omega_1$, while the information about the other component is lost. To see this, note that for all $x<1/(\tau_2\gamma_Y)$, $y>q_{\gamma_Y}$ and sufficiently large $t$
  \begin{eqnarray*}
    \lefteqn{tP\Big\{\frac{X-(\omega_2-g_2(U_Y(t/2)))}{\gamma_Y\tau_2g_2(U_Y(t/2))}>x \text{ or } \frac{Y-U_Y(t)}{\gamma_Y U_Y(t)}>y\Big\}} \\
     & = & \frac t2 \big( P\big\{g_1(Y)< \omega_1-\omega_2+g_2(U_Y(t/2))(1-\gamma_Y\tau_2x) \text{ or } Y>U_Y(t)(1+\gamma_Yy)\big\} \\
     & & \hspace{1cm} + P\big\{g_2(Y)<g_2(U_Y(t/2))(1-\gamma_Y\tau_2x) \text{ or } Y>U_Y(t)(1+\gamma_Yy)\big\}\big)\\
     & = & \frac t2 \big( P\{Y>U_Y(t)(1+\gamma_Yy)\}\\
      & & \hspace{1cm}+ P\big\{Y>\min\big(U_Y(t/2)(1-\gamma_Y\tau_2x)^{-1/\tau_2}(1+o(1)), U_Y(t)(1+\gamma_Yy)\big)\big\}\big) \\
     & \to & \frac 12 (1+\gamma_Yy)^{-1/\gamma_Y} + \max\Big((1-\gamma_Y\tau_2x)^{1/(\gamma_Y\tau_2)}, \frac 12 (1+\gamma_Yy)^{-1/\gamma_Y}\Big)\\
     & = & \mu\big(([-\infty,\infty]\times [q_{\gamma_Y},\infty])\setminus ([-\infty,x]\times [q_{\gamma_Y},y])\big).
  \end{eqnarray*}
  where $\mu$ denotes the measure concentrated on $\{(\gamma_Y^{-1}\tau_2^{-1}(1-t^{-\tau_2}),(t2^{-\gamma_Y}-1)/\gamma_Y)|t>0\}\cup \{(-\infty,(t-1)/\gamma_Y)|t>0\}$ such that $2\mu\big(\{(-\infty,(t-1)/\gamma_Y)|t>r\}\big)=\mu\big(\{(\gamma_Y^{-1}\tau_2^{-1}(1-t^{-\tau_2}),(t2^{-\gamma_Y}-1)/\gamma_Y)|t>r\}\big)=r^{-1/\gamma_Y}, r>0$.\qed
\end{example}
Similar effects occur if, for large values of $Y$, the random variable $X$ may attain values in the neighborhood of different points (possibly including $\infty$ and $-\infty$). The behavior of $X$ in the neighborhood of each finite point can be captured by a CEVM only if the conditions (ii) and (ii*) of Definition \ref{def:CEV}, which rule out mass on $\{\pm\infty\}\times \bar E^{(\gamma_Y)}$, are dropped. Then in general the limit measure is not unique
anymore (up to shift and scaling). In Example \ref{ex:mixt}, one would get three fundamentally distinct limits: the one derived in the example, and two further limit measures which describe the behavior of $X$ near $\omega_1$ and $\omega_2$, respectively, on a more detailed scale.  The two latter ones, which arise if one uses e.g.\ the normalizing functions $\alpha_i(t)=g_i(U_Y(t))$ and $\beta_i(t)=\omega_i$, have mass on $\{\infty\}\times \bar E^{(\gamma_Y)}$ and $\{-\infty\}\times \bar E^{(\gamma_Y)}$, respectively.

In general, of course, there may be arbitrarily many (even an infinite countable number of) accumulation points of $X$ for large values of $Y$, and hence of possible limits. Moreover, even if there is only one point of accumulation, different multiplicative normalizing functions $\alpha$ may lead to different limits if mass on $\{\pm\infty\}\times \bar E^{(\gamma_Y)}$ is allowed{; see e.g.\ Example 5.4 of \cite{DMR13}. So while in some cases omitting the condition (ii) in Definition \ref{def:CEV} allows for a more detailed analysis of $X$ for large values of $Y$, one pays the price of an enormously increased complexity.

An alternative ad-hoc solution in the case when $X$ is concentrated on several accumulation points $\omega_i$ may be to separately consider a CEVM (incl.\ condition (ii*)) for $X$ restricted to a neighborhood $N_i$ of each of these points, that is, to require vague convergence of $tP\big(\{((X-\beta_i(t))/\alpha_i(t),$ \linebreak[4]$(Y-d(t))/c(t))\in \cdot \}\cap\{X\in N_i\}\big)$ for each $i$. One would then have two layers of CEVM. The first is obtained by the present definition. If the marginal measure corresponding to the first coordinate is discrete (or, more generally, has a discrete part), then in addition one may consider `localized' CEVM models on suitable neighborhoods of each point of mass. Although such an approach would yield a much more complete picture, it seems practically feasible only in those cases when the number of accumulation points is small.

\begin{acknowledgement} This project was partly supported by the German Research Foundation DFG, Grant no JA 2160/1. We thank Juta Vollst\"{a}dt for fruitful discussions. The constructive and helpful remarks by a referee and an associate editor led to an improvement of the presentation.
\end{acknowledgement}

\end{document}